\documentclass{emsprocart}
\usepackage[all]{xy}
\usepackage{eucal}
\SelectTips{eu}{}
\contact[e-mail address]{Department of Mathematics, University of Washington, Box 354350, Seattle, USA. julia@math.washington.edu}

\newtheorem{theorem}{Theorem}[section]

\newtheorem{proposition}[theorem]{Proposition}

\theoremstyle{definition}
\newtheorem{definition}[theorem]{Definition}
\newtheorem{remark}[theorem]{Remark}
\newtheorem{example}[theorem]{Example}
\newtheorem{notation}[theorem]{Notation}

\numberwithin{equation}{section}

\def\Lie{\operatorname{Lie}\nolimits}
\def\JType{\operatorname{JType}\nolimits}
\def\Spec{\operatorname{Spec}\nolimits}
\def\Specm{\operatorname{Specm}\nolimits}
\def\Proj{\operatorname{Proj}\nolimits}

\def\Grass{\operatorname{Grass}\nolimits}
\def\LG{\operatorname{LG}\nolimits}
\def\Rad{\operatorname{Rad}\nolimits}
\def\Soc{\operatorname{Soc}\nolimits}
\def\Ker{\operatorname{Ker}\nolimits}
\def\Im{\operatorname{Im}\nolimits}

\def\Coh{\operatorname{Coh}\nolimits}

\def\HHH{\operatorname{H}\nolimits}
\def\Ext{\operatorname{Ext}\nolimits}
\def\Ann{\operatorname{Ann}\nolimits}
\def\End{\operatorname{End}\nolimits}
\def\Hom{\operatorname{Hom}\nolimits}
\def\PHom{\operatorname{PHom}\nolimits}
\def\SL{\operatorname{SL}\nolimits}
\def\GL{\operatorname{GL}\nolimits}

\def\Adj{\operatorname{Adj}\nolimits}

\def\rk{\operatorname{rk}\nolimits}
\def\dim{\operatorname{dim}\nolimits}
\def\stmod{\operatorname{stmod}\nolimits}
\def\ad{\operatorname{ad}\nolimits}
\def\res{\operatorname{res}\nolimits}
\def\id{\operatorname{id}\nolimits}

\newcommand{\cKer}{\mathcal K\text{\it er}\,}
\newcommand{\cIm}{\mathcal I\text{\it m}\,}
\newcommand{\cCoker}{\mathcal C\text{\it oker}\,}

\newcommand{\bA}{\mathbb A}
\newcommand{\bE}{\mathbb E}
\newcommand{\bF}{\mathbb F}
\newcommand{\bG}{\mathbb G}
\newcommand{\bP}{\mathbb P}
\newcommand{\Z}{\mathbb Z}
\newcommand{\cC}{\mathcal C}

\newcommand{\cF}{\mathcal F}
\newcommand{\cH}{\mathcal H}
\newcommand{\cI}{\mathcal I}

\newcommand{\cN}{\mathcal N}
\newcommand{\cO}{\mathcal O}

\newcommand{\cU}{\mathcal U}

\newcommand{\fg}{\mathfrak g}

\newcommand{\fu}{\mathfrak u}
\newcommand{\gl} {\mathfrak {gl}}
\newcommand{\fsl} {\mathfrak {sl}}
\newcommand{\fso} {\mathfrak {so}}
\newcommand{\fsp} {\mathfrak {sp}}

\newcommand{\wt}{\widetilde}
\newcommand{\bu}{\bullet}

\title[Representations and cohomology of finite group schemes]{Representations and cohomology of finite group schemes}

\author[Julia Pevtsova]{Julia Pevtsova\thanks{partially supported by the National Science Foundation award DMS-0953011}}

\begin{document}

\begin{abstract} 
This is a survey article covering developments in representation theory of finite group schemes over the last fifteen years. 
We start with the finite generation of cohomology of a finite group scheme and proceed to discuss various consequences and 
theories that ultimately grew  out of that result. This includes the theory of one-parameter subgroups and rank varieties for 
infinitesimal group schemes; the  $\pi$-points and  $\Pi$-support spaces for finite group schemes, modules of constant rank and constant Jordan type, and construction of bundles on varieties closely related to $\Proj \HHH^\bu(G,k)$ for an infinitesimal group scheme $G$. The material is mostly complementary to the article of D. Benson on elementary abelian $p$-groups in the same volume; we concentrate on the aspects of the theory which either hold generally for any finite group scheme or are specific to finite group schemes which are not finite groups. In the last section we discuss varieties of elementary subalgebras of modular Lie algebras, generalizations of modules of constant Jordan type, and new constructions of bundles on projective varieties associated to a modular Lie algebra. 
\end{abstract}



\maketitle

%\tableofcontents

\section{Introduction}
This survey  is based on the lecture given by the author at the International Conference on Representations of Algebras  in Bielefeld in 2012. Jon Carlson gave a talk on this subject at ICRA XII in Toru\'n in 2007 with the survey \cite{Car08} published in the same series.  In the current article we try to pick up where Carlson left off although some overlaps to set the stage were unavoidable.  We also focus on the general case of a finite group scheme as opposed to a finite group highlighted in \cite{Car08}.  

Even though this is an article about finite group schemes which is a much more general class than finite groups, we start by recalling the foundations of the classical theory of support varieties which motivated later developments featured in this survey.    In his seminal papers \cite{Q}, Quillen gave a general description of the maximal ideal spectrum of the commutative algebra $\HHH^{ev}(G; k)$ for a finite group $G$ in terms
of the elementary abelian $p$-groups of $G$. Following work of Alperin and Evens \cite{AE}, Carlson extended Quillen's work to introduce and study the {\it support variety} of a finite dimensional $G$-module, a closed conical subvariety of $\Spec \HHH^{ev}(G; k)$ which became an important invariant of modular representations. An essential feature of the support variety of a module of an elementary abelian $p$-group $E$, conjectured by Carlson  \cite{Car} and proved by Avrunin-Scott \cite{AS}, is that it admits a description in terms of cyclic shifted subgroups of the group algebra $kE$ with no recourse to cohomology. Such a description is referred to as ``rank variety".   

Representations of a finite group scheme  are equivalent to modules of its group algebra (also known as the algebra of measures) which is a finite dimensional cocommutative Hopf algebra. Hence, from the representation theoretic point of view, finite group schemes fall between finite groups and finite dimensional Hopf algebras. Some of the formalism involving cohomology, such as the definition of support variety originally developed for finite groups, works equally well for any finite dimensional Hopf algebra.  But any argument that uses reduction to Sylow subgroups or any deeper ``local-to-global" methods from finite group theory usually fails miserably for other classes of finite group schemes.  From the point of view of a Hopf algebraist this is not surprising: group algebras of finite groups are generated by the group-like elements whereas group algebras corresponding to other examples of finite group schemes, such as restricted enveloping algebras, can be primitively generated. Hence, they can exhibit behaviour very different from finite groups.  A case in point is finite generation of the cohomology algebra of a finite group scheme with coefficients in a field $k$ of positive characteristic $p$.  For finite groups this was known since late 50s - early 60s, the reduction to  a Sylow subgroup being a quick and easy step in the algebraic proof due to L. Evens \cite{E}.   For restricted Lie algebras finite generation of cohomology was proven in the 80s (\cite{FPar2}, see also \cite{AJ})  with the maximal unipotent  subalgebra playing a rather different role from a Sylow $p$-subgroup and the passage from the unipotent to the general case being highly non-trivial. The finite generation of cohomology of any finite group scheme was proven in a celebrated theorem of Friedlander and Suslin in 1995 (\cite{FS}).  Whether  the cohomology ring of any finite dimensional Hopf algebra with trivial coefficients  is finitely generated is still an open problem.   

The Friedlander-Suslin theorem on finite generation of cohomology unlocked the way for the applications of powerful geometric machinery to the study of representations and cohomology of a finite group scheme as pioneered by Quillen for finite groups Suslin-Friedlander-Bendel developed a theory of rank varieties for infinitesimal group schemes based on one-parameter subgroups. They identified the spectrum of cohomology with the variety of one-parameter subgroups, and proved an analogue of the Avrunin-Scott theorem. Friedlander and the author subsequently unified the existing theories for finite groups and infinitesimal group schemes  by introducing the notion of a $\pi$-point and $\Pi$-space. One interpretation of the main theorem of the theory of $\pi$-points, Theorem~\ref{thm:main}, is that it generalizes Quillen stratification and the Avrunin-Scott theorem to any finite group scheme.  

A byproduct of the theory of $\pi$-points was the discovery of modules of constant Jordan type (\cite{CFP}) which was an emerging subject at the time of publication of Carlson's article \cite{Car08}. There is now extensive literature on modules of constant Jordan type  some of which is mentioned in the references although it seemed impossible to list every contribution. In particular, D. Benson is currently writing a book on the subject which is available in electronic form. 

The development that followed modules of constant Jordan type was construction of vector bundles associated to these modules. Somewhat surprisingly, the roles   of different classes of finite group schemes got reversed at this stage. The more geometric nature of infinitesimal group schemes allows for construction of vector bundles on projective varieties of one-parameter subgroups whereas it is unknown whether an analogous construction exists for finite groups.  One may ask ``what about elementary abelian $p$-groups?", the subject of \cite{Ben2}. As we argue in Example~\ref{ex:elem}, from the representation-theoretic perspective, an elementary abelian  $p$-group $E$ is a very special case of an infinitesimal group scheme. Hence, constructions of \cite{FP12} apply. Focusing  on the case of an elementary abelian $p$-group brings considerable advantages: for example, $\Proj \HHH^{\rm ev}(E,k)$ is  a projective space which provides for many special geometric tools and leads to stronger results than in the general case of an infinitesimal group scheme.  We refer the reader to D. Benson's article for material specific to the elementary abelian case. 

The organization of the paper follows the evolution of the subject. In the first two sections we recall the definitions and basic properties of finite group schemes, their representations, and cohomology. We give an overview of the scheme of one-parameter subgroups $V(G)$ for an infinitesimal group scheme $G$ which plays an important role in the construction of the global $p$-nilpotent operator $\Theta$ in Section~\ref{sec:theta}. Section~\ref{sec:pi} is a dense summary of the theory of $\pi$-points with applications to classification of thick tensor ideal subcategories of $\stmod G$ and determination of the representation type of $G$. In Section~\ref{sec:jordan}
we introduce modules of constant rank and constant Jordan type  and construct non-maximal rank varieties.  This section  is somewhat  skimpy and short on examples due to the fact that the material is covered in Benson's article in the same volume.   One can also find more on the properties of non-maximal support varieties and modules of constant Jordan type in Carlson's article \cite{Car08} as well as the original papers. 

In Section~\ref{sec:theta} we sketch the construction of the global nilpotent operator $\Theta$ that acts on $M \otimes k[V(G)]$ for any module $M$ of an infinitesimal group scheme $G$. 
The operator $\Theta$ allows us to associate coherent sheaves on  $\Proj  V(G)$  (which is homeomorphic to $\Proj \HHH^\bu(G,k)$) to any representation of $G$. For representations of constant rank the associated sheaves turn out to be vector bundles.  
In the same section we list some properties of these bundles and calculations for the restricted Lie algebra $\fsl_2$.

In the final section we summarize the latest results of Carlson, Friedlander and the author  from \cite{CFP2}, \cite{CFP3}. We focus on a restricted Lie algebra $\fg$ and replace the variety $\Proj V(G)$ with the variety $\bE(r, \fg)$ which is the moduli space of the ``elementary Lie subalgebras"   of $\fg$  of dimension $r$. If $V(G)$ was the variety of one-parameter subgroups, then $\bE(r,\fg)$ can be viewed as the variety of ``$r$-parameter" subalgebras. These varieties are interesting geometric invariants of  $\fg$ on their own right. The beginning of Section~\ref{sec:elem} provides some information on what is known about them; many questions remain unanswered. Using restrictions to elementary subalgebras instead of one-parameter subgroups, we define modules of constant $(r,j)$-radical or socle rank (or type) that generalize modules of constant rank (or constant Jordan type).  Replacing the operator $\Theta$ of Section~\ref{sec:theta} with a more sophisticated construction we can associate vector bundles on $\bE(r, \fg)$ to modules of constant $(r,j)$-radical or socle rank. We finish the article with examples to illustrate our constructions.  

Throughout, $k$ is assumed to be an algebraically closed field of characteristic $p$. The assumption of algebraically closed can often be relaxed if one is willing  to work with schemes instead of varieties but we keep it to streamline the exposition. All vector spaces, algebras, and schemes are $k$-spaces, $k$-algebras, and $k$-schemes unless explicitly specified otherwise. The dual vector space of $M$ is denoted $M^\#$. For a field extension $K/k$, and a vector space $M$, we denote by $M_K=M \otimes_k K$ the base change from $k$ to $K$. For a scheme $X$, we set $X_K = X \times_{\Spec k} \Spec K$. 

\vspace{0.1in} 
The author would like to thank the organizers, particularly Henning Krause and Rolf Farnsteiner, for the invitation to the meeting and for letting her experience such a unique and invigorating event as ICRA. She would also like to thank Eric Friedlander for introducing her to the subject and for generously sharing his ideas over the years that they have worked on it together.

\section{Finite group schemes: definitions and examples} 
\label{sec:fgs} 
To define finite group schemes we take the ``functor of points" approach to the definition of an affine scheme such as the one taken in \cite{W}, \cite{Jan}, or \cite[Appendix]{EH}. We follow the terminology in \cite{W}.  

An affine scheme $X$ over a field $k$ is a representable $k$-functor from the category of commutative $k$-algebras to sets.  If $X$ is represented by the algebra $k[X]$, then as a functor it is given by the formula 
\[X(A) = \Hom_{k-alg}(k[X], A)\]
for any commutative $k$-algebra $A$.  The algebra $k[X]$ is the {\it coordinate algebra} or the {\it algebra of regular functions} of $X$. An affine scheme $X$ is {\it algebraic} if the coordinate algebra $k[X]$ is a finitely generated algebra.   By the Yoneda lemma, we have an anti-equivalence of categories:

\[\left\{\begin{array}{c}\text{affine algebraic}\\\text{schemes} \end{array}\right\} \quad \sim^{\rm op} \quad \left\{\begin{array}{c}\text{finitely generated}\\\text{commutative algebras} \end{array}\right\}\]

An {\it affine group scheme} $G$  is an affine scheme that takes values in the category of groups, that is, a representable functor from commutative $k$-algebras to groups. The coordinate algebra $k[G]$ is then a Hopf algebra and  the anti-equivalence of categories above restricts to the following:

\[\left\{\begin{array}{c}\text{affine algebraic}\\\text{group schemes} \end{array}\right\} \quad \sim^{\rm op} \quad \left\{\begin{array}{c}\text{finitely generated commu-}\\\text{tative Hopf algebras} \end{array}\right\}\]

For the purposes of this paper, all group schemes   are affine algebraic.   Moreover,  our main interest lies within the following class:

\begin{definition} 
A group scheme $G$ is {\it finite} if the coordinate algebra $k[G]$ is a finite dimensional $k$-algebra.
\end{definition} 
\noindent We have yet another anti-equivalence:
\begin{equation} 
\label{anti-equi}
\left\{\begin{array}{c}  \text{finite group}\\ \text{schemes} 
\end{array}\right\} \quad \sim^{\rm op} \quad  \left\{\begin{array}{c}  \text{finite dimensional {\bf commu-}}\\ \text{{\bf tative} Hopf algebras} 
\end{array}\right\}
\end{equation} 
 Motivated by Example~\ref{ex:finite} below, we use the notation 
 \[kG:=k[G]^\#\]
 and refer to this algebra as the {\it group algebra} of $G$.  Dualizing  the  anti-equivalence~\eqref{anti-equi} we arrive at the equivalence of categories that underlines the entire paper: 
 
\begin{equation} 
\label{equi}
\left\{\begin{array}{c}  \text{finite group}\\ \text{schemes} 
\end{array}\right\} \quad \sim \quad  \left\{\begin{array}{c}  \text{finite dimensional {\bf co-}}\\ \text{{\bf commutative} Hopf algebras} 
\end{array}\right\}
\end{equation}
%The group algebra $kG$ is also known as the {\it distribution algebra} or the {\it hyperalgebra} of $G$.    

One can define representations of a finite group scheme $G$ using the functorial perspective (see, for example, \cite[Ch. 3]{W}). The equivalence \eqref{equi} will then imply that there is an equivalence between the category of representations of $G$  and the category of $kG$-modules. In what follows, we shall use ``representation of $G$" and ``$kG$-module" interchangeably.

Since the category of $kG$-modules is abelian with enough injectives,  we can consider $\Ext^n_{G}(M,N) :=\Ext^n_{kG}(M,N)$, the $\Ext$-groups in the category of $kG$-modules for $n$ a non-negative integer. As usual, we set 
\[
\HHH^*(G,k) := \Ext^*_G(k,k) = \bigoplus\limits_{n \geq 0} \Ext^n_G(k,k), \quad   \HHH^*(G,M) := \Ext^*_G(k,M)
\]
where $k$ is the trivial module for $kG$ given by the augmentation (equivalently, counit)  map, and $M$ is any $kG$-module.

Since $kG$ is a finite dimensional Hopf algebra, it is self-injective (\cite{LS}, see also \cite[ch.5]{Jan}). Therefore, the category of $kG$-modules is Frobenius, that is, injective $kG$-modules are projective and vice versa (\cite{FW}). 

Due to the presence of the Hopf algebra structure on $kG$, there are two ways of defining the product on the cohomology algebra $\HHH^*(G,k)$: the cup product and the Yoneda product. The cup product is defined by tensoring projective resolutions and composing with the diagonal approximation map whereas the Yoneda product is given by splicing extensions. The existence and compatibility of these two products leads to the graded commutativity of $\HHH^*(G,k)$ by using a modified version of the classical {\it Eckmann-Hilton argument}.  One can find details, for example, in \cite[\S 3]{MPSW}.  

\begin{theorem}\label{th:graded}  
Let $G$ be a finite group scheme.  Then $\HHH^*(G,k)$ is a graded commutative $k$-algebra.  
\end{theorem}
\noindent We shall use the following notation throughout the paper:
$$
\HHH^\bu(G,k)=
\begin{cases}
\bigoplus\limits_{n \geq 0} \HHH^{2n}(G,k) & p>2 \\
\HHH^*(G,k) & p=2.
\end{cases}
$$
Theorem~\ref{th:graded} implies that $\HHH^\bu(G,k)$ is a commutative algebra. We next recall the key result of Friedlander and Suslin that allows one to employ techniques from algebraic geometry to study cohomology and representations of $G$. 

\begin{theorem}[Main theorem, \cite{FS}]\label{thm:fg} Let $G$ be a finite group scheme, and $M$  be a finite dimensional $kG$-module.  Then $\HHH^*(G,k)$ is a finitely generated $k$-algebra, and $\HHH^*(G,M)$ is a finitely generated module over $\HHH^*(G,k)$. 
\end{theorem}

We make some historical remarks about the finite generation of the cohomology  ring.   For $G$ a finite group with characteristic $k$ dividing the order of $G$, the finite generation of $\HHH^*(G,k)$ was proven by Golod \cite{Go}, Venkov \cite{V}, and Evens \cite{E} independently with Evens using purely algebraic techniques and adding the result about the module $\HHH^*(G,M)$.  It took another 35 years to prove the result for any finite group scheme.  It is conjectured that $\HHH^*(A,k)$  is finitely generated for {\it any} finite dimensional Hopf algebra $A$ (see, for example, \cite{EO} where it is formulated in the context of tensor categories), with partial progress made in \cite{GK}, \cite{BNPP}, and \cite{MPSW}. 

\vspace{0.1in} 
We finish this section with a definition of a (cohomological) support variety. Let $M$  be a finite dimensional $kG$-module. Then $\HHH^\bu(G,k) \simeq \Ext_G^\bu(k,k)$  acts on $\Ext_G^*(M,M)$ via Yoneda product: namely, for any extension $\xymatrix@=4mm{k \ar[r] & \ldots \ar[r] & k }$ and   
$\xymatrix@=4mm{M \ar[r] & \ldots \ar[r] & M }$ we tensor the first extension with $M$ and then concatenate.  This defines a pairing 
\[ \xymatrix@=6mm{\Ext_G^n(k,k) \times \Ext_G^m(M,M) \ar[r] & \Ext^{n+m}_G(M,M)} \] 
and then the (algebra) action
\[\xymatrix@=6mm{\Ext_G^\bu(k,k) \times \Ext_G^*(M,M) \ar[r] & \Ext_G^*(M,M)} \] 
(we refer the reader to \cite[I.2.6]{Ben}  for details.  See also \cite[I.3.2]{Ben}  where the same action is expressed in terms of cup product). 

\begin{definition}
\label{defn:supp}  Let $\cI_M = \Ann_{\HHH^\bu(G,k)}\Ext_G^*(M,M)$  be the annihilator  ideal of $\Ext_G^*(M,M)$ as a module over $\HHH^\bu(G,k)$.  
The (cohomological) support variety  of $M$, denoted $|G|_M$,  is a closed subset of $\Spec \HHH^\bu(G,k)$ defined by the  ideal $\cI_M$. 
It has a canonical structure of an affine variety corresponding to the reduced affine scheme $\Spec (\HHH^\bu(G,k)/\cI_M)_{\rm red}$.
 %\[|G|_M = \Specm \HHH^\bu(G,k)/\Ann_{\HHH^\bu(G,k)}\Ext_G^\bu(M,M)\]  
\end{definition}   
The algebra $\HHH^\bu(G,k)$ is graded and the ideal $\cI_M$ is homogeneous. Hence, we can consider the ``projectivized" support variety of $M$. We denote by $\Proj |G|_M$ the closed points of the projective scheme $\Proj \left(\HHH^\bu(G,k)/\cI_M\right)$, that is, the subset of homogeneous prime ideals in $\HHH^\bu(G,k)$ of dimension one that contain $\cI_M$.

\subsection*{Examples of finite group schemes} 

\begin{example}[Finite groups]\label{ex:finite}  Let $G$   be a finite group and $kG$ be the group algebra. We define the associated {\it constant group functor} $\wt G$ as follows: for a commutative $k$-algebra $R$, 
\[
\xymatrix{\wt G(R) \ar@{=}[r]^-{\rm def} &  G^{\times |\pi_0(R)|}.} 
\]
where $\pi_0(R)$ is the set of connected components of $\Spec R$.
We have $k[\wt G] \simeq k^{\times |G|}$, and $k\wt G \simeq k G$. 
\end{example}

\begin{example}[Frobenius kernels]\label{ex:frob} Let 
$\xymatrix{f:k \ar[r]^-{\lambda \mapsto \lambda^p}& k}$ be the Frobenius map.  For a commutative $k$-algebra $A$ we define its {\it Frobenius twist} as a base change over the Frobenius map: $A^{(1)} :=A \otimes_f k$.  There is a $k$-linear algebra map  $F_A: A^{(1)} \to A$ given by $F_A(a \otimes \lambda)=\lambda a^p$. If $G$ is a group scheme, then the Frobenius twist $k[G]^{(1)} = k[G] \otimes_f k$  is again a  Hopf algebra over $k$ and, therefore, defines another group scheme $G^{(1)}$ which we call the {\it Frobenius twist} of $G$. The algebra map $F_{k[G]}: k[G^{(1)}] = k[G]^{(1)}  \to k[G]$ induces the {\it{Frobenius}} map of group schemes 
\[F = F_G: G \to G^{(1)}.\]   
\begin{definition}  The $r^{\rm th}$ Frobenius kernel of a group scheme $G$ is the \it {group scheme theoretic} kernel  of $r$ iterations of the Frobenius map:
\[ G_{(r)} = \Ker F^{r}: G \to G^{(r)}.\] 
\end{definition}

\begin{definition} A finite group scheme $G$ is called infinitesimal if  the coordinate algebra $k[G]$ is local. 
An infinitesimal finite group scheme $G$ is of height $\leq r$ if for any $x \in k[G]$ not a unit, $x^{p^r}=0$. An infinitesimal group scheme of height $r$ is an infinitesimal group scheme of height $\leq r$ but not of height $\leq r-1$. 
\end{definition} 
Frobenius kernels are examples of infinitesimal  group schemes. In particular, they are highly non-reduced: for any field extension $K/k$  there is only one $K$-point in $G_{(r)}$.  

If the group scheme $G$ is defined over the prime field $\bF_p$, there is an isomorphism $G^{(1)} \simeq G$.  By postcomposing with this isomorphism, we can consider the Frobenius map as a self map: $G \to G$.  We describe this explicitly for a particular example of $\GL_n$.  
\end{example}

\begin{example}[$\GL_n$]
Define the algebraic group (scheme) $\GL_n$ as
$$\GL_n(R) = \{\text{invertible } n \times n \text{ matrices over } R \}$$
for any commutative $k$-algebra $R$. Then $k[\GL_n] \simeq k[X_{ij}, \frac{1}{\det}]$, $1 \leq i,j \leq n$, with the Hopf algebra structure defined as follows: 

\vspace{0.1in}
\begin{tabular}{lll}{\it coproduct } & $\nabla: k[\GL_n] \to k[\GL_n] \otimes k[\GL_n]$ &  $\nabla(X_{ij}) = \sum\limits_{\ell=1}^n X_{i\ell}\otimes X_{\ell j}$,\\
{\it counit } & $\epsilon: k[\GL_n]\to k$ & $\epsilon(X_{ij})=\delta_{ij}$,\\
{\it antipode} & $S: k[\GL_n] \to k[\GL_n]$ & $S(X_{ij}) = \frac{\Adj(X_{ij})}{\det}$ 
\end{tabular}, 

\vspace{0.1in}
\noindent where $\Adj(X_{ij})$ is the {\it adjoint} matrix to $(X_{ij})$.
The Frobenius map is given explicitly by the formula:
\[
\xymatrix{F: \GL_n\ar[rr]^-{(a_{ij}) \to (a_{ij}^p)}&&  \GL_n}
\] 
Hence, as a group scheme, 
$$
\GL_{n(r)}(R)= \{(a_{ij})_{1\leq i,j\leq n}\, | \, a_{ij} \in R, a_{ij}^{p^r} = \delta_{ij} \}.
$$  
As an algebra,
$$k[\GL_{n(r)}] \simeq k[X_{ij}]/(X_{ij}^{p^r}-\delta_{ij}),$$
with the Hopf algebra structure inherited from $k[\GL_n]$. 

\end{example}

\begin{example}[$\bG_a$] The additive group $\bG_a$ is defined as follows:
\[\bG_a(R) := R^{+}, \text{ the additive group of the ring } R. \] 
We have $k[\bG_a] \simeq k[T]$ with the Hopf algebra structure determined by the following formulas:

\vspace{0.1in}
\begin{tabular}{lll}{\it coproduct } & $\nabla: k[T] \to k[T] \otimes k[T]$ &  $\nabla(T) = T \otimes 1 + 1 \otimes T$,\\
{\it counit } & $\epsilon: k[T]\to k$ & $\epsilon(T)=0$,\\
{\it antipode} & $S: k[T] \to k[T]$ & $S(T) = -T$. 
\end{tabular} 

\vspace{0.1in}
\noindent The Frobenius map is given by $\xymatrix{F(R): R^+\ar[r]^-{a \mapsto a^p}& R^+}$, and Frobenius  kernels have the following form:
\[
\bG_{a(r)}(R) = \{ a \in R \, | \, a^{p^r}=0\}.
\]
We have $k[\bG_{a(r)}] \simeq k[T]/T^{p^r}$ with the Hopf algebra structure inherited from $k[T]$, and 
\[
k\bG_{a(r)} \simeq k[u_0, \ldots, u_{r-1}]/(u_0^p, \ldots, u_{r-1}^p)
\]
 where the generators $u_0, \ldots, u_{r-1}$ are the linear duals to $T, T^p, \ldots, T^{p^{r-1}}$.

\end{example}

\begin{example}[Restricted Lie algebras] As we explain below, the example of restricted Lie algebras is a special case of Example~\ref{ex:frob}. 
\begin{definition} A Lie algebra $\fg$ over the field $k$ is called {\it restricted} if it is endowed with the $p^{\rm th}$ power operation 
$(-)^{[p]}: \fg \to \fg$ satisfying the following axioms: 

\vspace{0.1in} 
(1) $\ad(x^{[p]}) = \ad(x)^p$  for all $x \in \mathfrak g$,

(2)  $(\lambda x)^{[p]} = \lambda^px^{[p]}$ for $\lambda \in k$, $x \in \fg$, 

(3)  $(x+y)^{[p]} = x^{[p]} + y^{[p]} + \sum\limits_{i=1}^{p-1}\frac{s_i(x,y)}{i}$, where $s_i(x,y)$ 
is the coefficient of $t^{i - 1}$ in the formal expression $\ad(tx + y)^{p - 1}(x)$.
\end{definition}

Classical Lie algebras such as $\gl_n, \fsl_n, \fso_n, \fsp_{2n}$  defined over a field $k$ of characteristic $p$ are restricted, with the $[p]^{\rm th}$ power map being matrix multiplication. The Lie algebra $\Lie G$ of any group scheme $G$ over $k$  has a naturally defined $[p]^{\rm th}$ power operation (see \cite{W})  and, hence, is restricted.  

To any restricted Lie algebra $\fg$ over $k$ we associate its {\it restricted enveloping algebra}:
\[
\fu(\fg) = U(\fg)/ \langle x^p - x^{[p]}, x \in \fg \rangle
\] 
where $U(\fg)$  is the universal enveloping algebra of $\fg$. 
This is a finite  dimensional cocommutative Hopf algebra with the coproduct $\nabla: \fu(\fg) \to \fu(\fg) \otimes \fu(\fg)$ primitively generated; that is, $\nabla(x) =  x \otimes 1 + 1 \otimes x $ for any $x \in \fg$.  

A {\it restricted} representation of a restricted Lie algebra $\fg$ is a representation $M$ of $\fg$ as a Lie algebra which respects the $[p]^{\rm th}$ power structure. Namely, for $x\in \fg$, $m\in M$, 
$$
\underbrace{(x(x\ldots(x}_p m)\ldots) = x^{[p]}m.
$$ 
The structure of a restricted $\fg$ representation on a vector space $M$ is then equivalent to the structure of a $\fu(\fg)$-module. All Lie algebra representations in this paper are assumed to be restricted and will be referred to as $\fg$-modules. 

Using the equivalence \eqref{equi}, we can associate an infinitesimal group scheme of height 1 to any restricted Lie algebra $\fg$. Indeed, since $\fu(\fg)$ is a finite dimensional cocommutative Hopf algebra over $k$, its dual $\fu(\fg)^\#$ is a commutative finite dimensional Hopf algebra.  Hence, there exists a finite group scheme which we denote by $\wt \fg$ such that $k[\wt \fg] =  \fu(\fg)^\#$. One can show that $\wt \fg$ is necessarily an infinitesimal finite group scheme of height 1. Moreover, for any group scheme $G$, we have an isomorphism of  Hopf algebras
$k[G_{(1)}] \simeq \fu(\Lie G)^\#$. The associations $\fg \mapsto \wt \fg$ and $G \mapsto \Lie G$ define an equivalence of categories between the restricted Lie algebras and  finite infinitesimal  group schemes of height 1.  For any particular group scheme $G$ we have an equivalence of categories of representations of the Frobenius kernel $G_{(1)}$ and the restricted Lie algebra $\Lie G$:
\[\xymatrix{ \text{ Representations of } G_{(1)}  \ar@{<->}[r]^-\sim&  \fu(\Lie G) - {\rm mod}}.\]
(See \cite[II, \S 7, 3.9-3.12]{DG} for details.)
\end{example}

\begin{example}[Elementary abelian $p$-groups] \label{ex:elem}  Towards the end of the paper we focus on the case of a restricted Lie algebra or, equivalently,  finite infinitesimal group scheme of height 1. Here, we explain how from the point of view of representation theory elementary abelian $p$-groups can be considered in that framework. 

Let $E = {\Z/p}^{\times r}$ be an elementary abelian $p$-group of rank $r$. Let $\{g_1, \ldots, g_r\}$ 
be generators of $E$ and let $x_i = g_i-1 \in kE$. This choice of generators determines an isomorphism of algebras 
$$kE \simeq k[x_1, \ldots, x_r]/(x_1^p, \ldots, x_r^p).$$  

Let $\fg_a = \Lie \bG_a$ be the Lie algebra of the additive group scheme $\bG_a$. Then $\fu(\fg_a) \simeq k[t]/t^p$, and 
$\fu(\fg_a^{\oplus r}) \simeq k[x_1, \ldots, x_r]/(x_1^p, \ldots, x_r^p) \simeq kE$. Therefore, the categories of representations of $\fg_a^{\oplus r}$ and of an elementary abelian $p$-group of rank $r$ are equivalent as abelian categories. It should be noted that they do have different tensor products since the Hopf algebra structures on $kE$ and $\fu(\fg_a^{\oplus r})$ are different. 
\end{example}

%%%%%%%%%%%%%%%%%%%%%%%%%%%%%%%%%%%SECTION ON COHOMOLOGY%%%%%%%%%%%%%%%%%%%%%%%
\section{The spectrum of cohomology of a finite group scheme}
\label{sec:coh}
It would seem impossible to write about cohomology of finite group schemes without mentioning Quillen's fundamental stratification theorem for finite groups.  We start by recording the cohomology of an elementary abelian $p$-group. 
\begin{theorem}
\label{thm:elem} 
Let $E \simeq {\Z/p}^{\times r}$ be an elementary abelian $p$-group of rank $r$.  Then 
\[ 
\HHH^*(E,k)  \simeq \begin{cases}
k[x_1, \ldots ,x_r] \otimes \Lambda^*(y_1, \ldots, y_r) \text{ where } \deg x_i =2, \deg y_i=1 & p>2 \\
k[y_1, \ldots, y_r] \text{ where } \deg y_i=1 & p=2.
\end{cases}
\]
\end{theorem}
 \noindent Theorem~\ref{thm:elem} implies that for an elementary abelian $p$-group  $E$ of rank $r$ , 
\[\Spec \HHH^\bu(E,k) \simeq \bA^r,\] 
the affine $r$-space over $k$.

We give the Quillen stratification theorem here in the ``weak form" referring the reader to any of the excellent 
sources in the literature for further details (such as \cite[II.5]{Ben}). 

Let $G$ be a finite group.  For any elementary abelian $p$-subgroup $E$  of $G$ we have a restriction map on cohomology  $\res_{E, G} :\HHH^*(G, k) \to \HHH^*(E,k)$ and the induced map on spectra: $\res_{E, G}^*: \Spec \HHH^\bu(E, k) \to \Spec \HHH^\bu(G,k)$.  Then the weak form of Quillen stratification states that  
\[ \Spec \HHH^\bu(G,k) = \bigcup\limits_{E \subset G}  \res_{E, G}^*\Spec \HHH^\bu(E, k). \]
Moreover, 
\[\res_{E, G}^*\Spec \HHH^\bu(E, k) \simeq  (\Spec \HHH^\bu(E, k))/W_G(E),\]
where $W_G(E)$ is the  Weyl group of $E$ in $G$, that is, the normalizer of $E$ modulo the centralizer of $E$. 
Hence, geometrically $\Spec \HHH^\bu(G,k)$ is a union of finite quotients of affine spaces.  The ``strong form" of Quillen's theorem prescribes how this union is taken, or, equivalently, how the finite quotients of affine spaces stratify $\Spec \HHH^\bu(G,k)$.  We shall see next that for a restricted Lie algebra the situation is quite different.  

\vspace{0.1in}  Let $\fg$ be a restricted Lie algebra.  We denote by $\cN$ (or $\cN(\fg)$) the {\it nullcone} of $\fg$: the set of all nilpotent elements of $\fg$. 
It can be given the structure of a variety (or a reduced scheme) by considering it as a closed subset of the linear space $\fg \simeq \Spec S^*(\fg^\#)$. We also define the {\it restricted nullcone} of $\fg$ as the subset 
\[\cN_p = \{ x\in \fg \, | \, x^{[p]}=0 \} \] 
of $\fg$ consisting of all $[p]$-nilpotent elements.  Similarly to $\cN$, it has a structure of a variety considered as a closed subset of $\fg$.  

For a connected reductive group $G$ we denote by $h$ its Coxeter number. The close relationship between the nullcone and cohomology is revealed in the following theorem.  
\begin{theorem}[\cite{FPar}, \cite{AJ}]
Let $G$ be a connected reductive  group and assume that $p>h$. Let $\fg = \Lie G$.   Then $\HHH^{\rm odd}(\fg, k) = 0$ and 
\[\HHH^\bu(\fg,k)  \simeq k[\cN]^{(1)}.\]
\end{theorem}
\noindent Even though the isomorphism between algebras fails in general, the following theorem holds for any restricted Lie algebra:
\begin{theorem}[\cite{SFB2}]
\label{thm:lie}  For any restricted Lie algebra $\fg$ there is a homemorphism of varieties: 
\[\Spec \HHH^\bu(\fg,k) \simeq \cN_p.\]
\end{theorem}
In fact, the map in the last theorem is more than a homeomorphism: the inverse is induced by a map of algebras $\psi: \HHH^\bu(\fg,k) \to k[\cN_p]$ which is an isogeny (or an $F$-isomorphism as defined by Quillen). Namely, it has a nilpotent kernel and the image contains the $p$-th power of any element. 

The theorem identifying the spectrum of cohomology of an arbitrary restricted Lie algebra is a special case of the theory constructed by Suslin-Friedlander-Bendel \cite{SFB1}, \cite{SFB2} for infinitesimal group schemes that we now briefly recall. 

For any affine group scheme $G$ over $k$ we consider the functor 
\begin{equation}
\label{def:functor} 
 R \mapsto \Hom_{\rm grp. sch}( \bG_{a(r),R}, G_R) 
 \end{equation}
sending a commutative $k$-algebra to the set of homomorphisms of group schemes over $R$.    A map of group schemes of the 
form $\bG_{a(r),R} \to G_R$ is referred to as a {\it  one-parameter subgroup} of $G_R$. 
\begin{theorem}[\cite{SFB1}]\label{thm:repr}    The functor \eqref{def:functor} is representable by an affine scheme $V_r(G)$ with the coordinate algebra $k[V_r(G)]$.  Equivalently, there is a natural isomorphism 
\[\Hom_{\rm grp. sch}( \bG_{a(r),R}, G_R)  \simeq \Hom_{k-\rm alg}(k[V_r(G)], R).\]
\end{theorem}

\begin{example}  $[1].$ Let $r=1$. Then the $k$-points of the scheme $V_1(G)$ coincide with the restricted nullcone of $\Lie G$: $V_1(G)(k) \simeq \cN_p(\Lie G)$. \\[1pt]
$[2].$ For $\GL_n$, $V_r(\GL_n)$ is the scheme of $r$-tuples of $p$-nilpotent pairwise commuting matrices.   It was pointed out in \cite{SFB1} that this description can be extended to any reductive group of exponential type. In more recent work \cite{Mc}, \cite{Sob}, the identification of $V_r(G)$ with the scheme of $r$-tuples of  $[p]$-nilpotent commuting elements  of $\Lie G$ was further extended to any reductive group $G$ under the assumption that $p$ is at least the Coxeter number for $G$. 
\end{example}

For an infinitesimal  group scheme of height $r$, we drop the subscript $r$ in $V_r(G)$ and use the notation  $V(G)$.  The main results of \cite{SFB1}, \cite{SFB2} amount to the following theorem. 

\begin{theorem}
\label{thm:inf}
%[\cite[5.2]{SFB1}] 
Let $G$ be an infinitesimal group scheme of height $r$. There exists a natural homomorphism of $k$-algebras
\[\psi:\HHH^\bu(G,k) \to k[V(G)]\] with  nilpotent kernel  whose image contains the $p^r$-th power of any elements of $k[V(G)]$.   Consequently, the induced map  of schemes $V(G) \to \Spec \HHH^\bu(G,k)$ is a $p$-isogeny and the corresponding map of varieties of closed points is a homeomorphism:
\begin{equation} 
\label{Psi}
\Psi: V(G)(k) \simeq \Specm \HHH^\bu(G,k), 
\end{equation} 
where $\Specm \HHH^\bu(G,k)$ is the maximal ideal spectrum of $\HHH^\bu(G,k)$.
\end{theorem}
\noindent Theorem~\ref{thm:lie}  is a special case of Theorem~\ref{thm:inf} for $r=1$ thanks to the correspondence between restricted Lie algebras and infinitesimal group schemes of height $1$. 

%%%%%%%%%%%%%%%%%%%%%%%%%%%%%%%%%%%%%%%%%%%Section on $\pi$-points%%%%%%%%%%%%%%%%%%%%%%%%%%%%%%%%%%%%%%%%%%%%5
\section{$\pi$-points} 
\label{sec:pi} 
In the previous section we described the spectrum of cohomology in terms of a ``model space" which was built from maps from some ``easy" objects to the group scheme: either elementary abelian $p$-subgroups for finite groups or one-parameter subgroups for infinitesimal groups.  We now explain a unified approach that works for any finite group scheme.  

We say that a map of algebras $\phi: A \to B$ is left (respectively, right) flat if  $\phi$ makes $B$ into a flat left (respectively, right) $A$-module. 
If $B$ is a  finite dimensional Frobenius algebra then any map of the form $\phi: k[x]/x^p \to B$ is left flat if and only if it is right flat (\cite{Sob}). Hence,  we do not distinguish between left and right flat in the definition below.

A finite group scheme $U$  over a field $K$ is abelian unipotent if the group algebra $KU$ is a local commutative $K$-algebra. Recall that for a field extension $K/k$, we denote by $KG=kG \otimes_k K$ the scalar extension from  $k$  to $K$. The finite group scheme over $K$ corresponding to the group algebra $KG$ is denoted $G_K$.    

\begin{definition}\label{defn:pi}
Let $G$ be a finite group scheme, and $K/k$ be a field extension.  A {\it $\pi$-point} $\alpha_K$ of $G$ is a flat map   of algebras $\alpha_K: K[x]/x^p \to KG$ such that there exists a  unipotent abelian subgroup scheme of $U \subset G_K$ and a commutative diagram
\[\xymatrix@=1.2mm{\alpha_K: K[x]/x^p \ar[rr]\ar[dr]&& KG\\
& KU \ar[ur] }\] 
where the map $KU \to KG_K$ is induced by the embedding $U \subset G_K$. 
\end{definition}
Note that in the definition a $\pi$-point, $\alpha_K$ is {\it only} a map of algebras whereas the map 
$KU \to KG$ is a map of Hopf algebras. We do not require that $U$ be defined over $k$. 

\begin{notation} For a $\pi$-point $\alpha_K:K[x]/x^p \to KG$, and a $G$-module $M$, we denote by $\alpha_K^*(M_K)$  the $K[x]/x^p$-module obtained by pulling back $M_K$  via $\alpha_K$.   
\end{notation} 

\begin{definition}{\cite{FP07}}\label{defn:equiv} Let $K, L$ be field extensions of $k$. Two $\pi$-points $\alpha_K: K[x]/x^p \to KG$, $\beta_L: L[x]/x^p \to LG$ are equivalent, written $\alpha_K \sim \beta_L$, 
if the following condition holds:  for any finite dimensional $kG$-module $M$, $\alpha_K^*(M_K)$ is free if and 
only if $\beta_L^*(M_L)$ is free (as $K[x]/x^p$ (respectively, $L[x]/x^p$)-module).
\end{definition}

\begin{example} 
\label{equiv:elem}
Let $E$ be an elementary abelian group of rank $r$ so that $kE \simeq k[x_1, \ldots, x_r]/(x_1^p, \ldots, x_r^p)$. Let $I = \Rad(kE)$ be the augmentation ideal. In this case the equivalence relation of Definition~\ref{defn:equiv} has a very explicit interpretation.  For simplicity, we restrict to $\pi$-points defined over the ground field $k$.  Let $\alpha, \beta: k[x]/x^p \to kE$ be two $\pi$-points (defined over $k$). Then $\alpha \sim \beta$ if and only if there exists a scalar $c \in k^*$ such that $(\alpha - c \beta)(t) \in I^2$ (see \cite[2.2, 2.9]{FP05}).
\end{example}

\begin{definition} For a finite group scheme $G$, we denote by $\Pi(G)$ the set of equivalence classes of $\pi$-points of  $G$: 
\[\Pi(G) = \frac{ \{ \pi-\text{points} \}}{\sim},\]
and by $[\alpha_K]$ the equivalence class of $\alpha_K$.  
For a $kG$-module $M$, we define the $\Pi$-support of $M$, 
\[\Pi(G)_M \subset \Pi(G),\] 
 as the subset of equivalence classes of $\pi$-points $\alpha_K$ such that $\alpha^*_K(M)$ is {\bf {not free}} as $K[x]/x^p$-module. 
\end{definition} 

Using $\Pi$-supports, we endow $\Pi(G)$ with Zariski topology in the following way:
\begin{proposition}[3.4 \cite{FP07}] Declaring 
\[ \{ \Pi(G)_M \subset \Pi(G) \, | \, M \text{ is finite dimensional } kG-\text{module}\}\]  
 to be the closed subsets gives $\Pi(G)$ a structure of a Noetherian topological space. 
\end{proposition}

The $\Pi$-supports of modules satisfy a number of nice properties which would be familiar to anyone who has studied support varieties.  One of them is the ``tensor product theorem". This property turns out to be surprisingly non-trivial due to the fact that $\pi$-points are 
not Hopf algebra maps and, hence, the restriction along a $\pi$-point does not commute  with the tensor product.  
%The $\Pi$-supports of modules satisfy multiple pleasing properties including the {\it tensor product theorem}.
\begin{theorem}{\cite[3.9]{FP05}, \cite[5.2]{FP07}} Let $G$ be a finite group scheme, and $M$, $N$  be $kG$-modules (not necessarily finite dimensional).  Then \[\Pi(G)_{M \otimes N} = \Pi(G)_M \cap \Pi(G)_N.\]
\end{theorem}

We define a continuous map 
\[\Psi: \Pi(G) \to \Proj \HHH^\bu(G,k)\]
in the following natural way:  Let $\alpha_K: K[x]/x^p \to KG$ be a $\pi$-point.  It induces a non-trivial (that is, not landing in degree 0) map on cohomology $\alpha_K^*: \HHH^\bu(G, K) \to \HHH^\bu(K[x]/x^p, K)$.    The kernel of $\alpha^*_K$ is a homogeneous ideal of $\HHH^\bu(G, K)$. Then $\Psi$  sends the equivalence class of $[\alpha_K]$ to the homogeneous prime ideal $\Ker \alpha_K^* \cap \HHH^\bu(G,k)$ where $\HHH^\bu(G,k)$ is embedded into $\HHH^\bu(G,K) = \HHH^\bu(G,k) \otimes K$
 by sending $\xi \in \HHH^\bu(G,k)$ to $\xi \otimes 1 \in \HHH^\bu(G,K)$:
 \[ \Psi([\alpha_K]) = \Ker \alpha_K^* \cap \HHH^\bu(G,k). \]  
 
\begin{theorem}[3.6 \cite{FP07}] 
\label{thm:main} 
For any finite group scheme $G$, $\Psi: \Pi(G) \to \Proj \HHH^\bu(G,k)$ is a homeomorphism.    For any finite dimensional 
$kG$-module $M$, $\Psi$ restricts to a homeomorphism  
$\Pi(G)_M \simeq \Proj |G|_M$.   
\end{theorem} 
\begin{remark}  There exists a scheme structure on $\Pi(G)$ defined solely in terms of representation theory which makes the homeomorphism $\Psi$ an isomorphism of schemes (see \cite[\S 7]{FP07}). 
\end{remark} 

\begin{example}  Let $E\simeq {\Z/p}^{\times r}$ be an elementary abelian $p$-group, and choose generators $\{g_1, \ldots, g_r\}$.  Then $kE \simeq k[x_1, \ldots, x_r]/(x_1^p, \ldots, x_r^p)$ with $x_i = g_i-1$ for $1 \leq  i \leq r$ being the generators of the augmentation ideal $I$. For any $\underline\alpha = (\alpha_1, \ldots, \alpha_r) \in \bA^r$, a {\it cyclic shifted subgroup} of $kE$ is a cyclic subgroup generated by $X_{\underline\alpha} +1$ where $X_{\underline\alpha} = \alpha_1 x_1 + \ldots \alpha_rx_r$. 
The space of all cyclic shifted subgroups plus $\{ 0 \}$ is naturally identified with $I/I^2 \simeq S^*((I/I^2)^\#) \simeq \bA^r$ and denoted $V(E)$. Then for any $kE$-module $M$, we have Carlson's rank variety $V(E)_M$  defined as follows: 
\[V(E)_M := \{ \underline\alpha \in V(E) \, | \, M\downarrow_{\langle X_{\underline \alpha}+1 \rangle} \text{ is not free } \} \cup \{ 0 \} \] 
The variety $V(E)_M$ is conical; we can consider the associated ``projectivized" rank variety $\Proj V(E)_M$ as a closed subset of $\Proj V(E) \simeq \bP^{r}$.   Example~\ref{equiv:elem} implies that we have a homeomorphism
\[\Proj V(E) \simeq \Pi(E)\] 
which restricts to 
\[\Proj V(E)_M \simeq \Pi(E)_M.\]
Hence, in the case of an elementary abelian $p$-group Theorem~\ref{thm:main} reduces to the ``projectivized" version of the
Carlson's conjecture proved by Avrunin and Scott (\cite{AS}): there is a homeomorphism $V(E) \simeq \Spec \HHH^\bu(E,k)$ which restricts to a homeomorphism between rank and support varieties: $V(E)_M \simeq |E|_M$. 
\end{example}

\begin{example}
For $G$ an infinitesimal group scheme, Theorem~\ref{thm:main} specializes to a ``projectivized" version of Theorem~\ref{thm:inf}. We describe how to go from a one-parameter subgroup to a $\pi$-point and refer the reader to \cite{FP07} for further details.  

\begin{notation}
\label{not:epsilon} Recall that $k\bG_{a(r)} \simeq k[u_0, \ldots, u_{r-1}]/(u_0^p, \ldots, u_{r-1}^p)$ where $u_i$ is the linear dual to $T^{p^i}$. 
We fix a map of algebras 
\[\epsilon:k[x]/x^p \to k\bG_{a(r)}\] 
which sends $x$ to $u_{r-1}$. 
\end{notation} 
Let $G$ be an infinitesimal group scheme of height $r$ and let $\mu: \bG_{a(r)} \to G$  be a one-parameter subgroup defined over the ground field $k$. It induces a map on group algebras $\mu_*: k\bG_{a(r)} \to kG$. Precomposing 
$\mu_*$ with $\epsilon$ we get a $\pi$-point $ \mu_* \circ \epsilon: k[x]/x^p \to kG$. The association
\begin{equation}\label{eq:pi} 
\mu \mapsto   \mu_* \circ \epsilon 
\end{equation}
determines a homeomorphism between the varieties of $k$-points of $\Proj V_r(G)$ and $\Pi(G)$.  

\end{example}

\subsection*{Applications} 
Since $kG$ is a Frobenius algebra for any finite group scheme $G$, we can associate to $G$ the  triangulated category $\stmod kG$, or, equivalently, $D_{\rm Sing}(G)$, the category of singularities of $kG$ (see, for example, \cite{Kel}; this is also discussed in \cite{Car08}). Objects of $\stmod kG$ are finite dimensional $kG$-modules, and Hom-sets are quotients defined as follows: 
\[ \Hom_{\stmod kG} (M,N) = \frac{\Hom_G(M,N)}{\PHom_G(M,N)},\] where $\PHom_G(M,N)$ are $G$ maps from $M$ to $N$ that factor  through a projective module. 

Let $T$ be a triangulated category. A {\it thick subcategory}  $\cC$ of $T$ is a full triangulated subcategory closed under direct summands. 
If $T$ is a symmetric monoidal triangulated category, then a thick subcategory $\cC$ is tensor ideal if for any $C \in  \cC$ and $X \in T$, $X \otimes C  \in \cC$. 

Using the properties of $\Pi$-supports, particularly the ``tensor product theorem", one can classify thick tensor ideal subcategories of $\stmod kG$ in terms of subsets  of $\Pi(G)$.  The following result, conjectured by Hovey-Palmiery-Strickland \cite{HP}, is a generalization to all finite group schemes of a theorem of Benson-Carlson-Rickard \cite{BCR}.  

\begin{theorem}[6.3 \cite{FP07}]\label{thm:class}  There is one-to-one order preserving correspondence between lattices of tensor ideal subcategories of $\stmod kG$ and subsets of $\Pi(G)$ closed under specialization.
\end{theorem}

For  a small symmetric monoidal triangulated  category $T$,  P. Balmer defined the spectrum of $T$, $\Spec T$, a ringed space 
that ``classifies" the thick tensor ideal subcategories of $T$.  Using Balmer's spectrum and the scheme structure on $\Pi(G)$ (\cite[7.5]{FP07}), Theorem~\ref{thm:class} can be given the following slick - and stronger - reformulation.
\begin{theorem}[\cite{Ba05}]\label{thm:balmer}
%[Balmer '05, Friedlander-P. '07] 
Let $G$ be a finite group scheme. There is an isomorphism  of schemes
\[\Spec (\stmod kG) \simeq \Proj \HHH^\bu(G,k).\]
\end{theorem}
An essential component of the theory of $\Pi$-supports is that it retains good properties even for infinite-dimensional modules - whereas the cohomological supports do not (in particular, the homeomorphism in Theorem~\ref{thm:main} breaks down for infinite dimensional modules).  This feature plays an important   role in the proofs of Theorems~\ref{thm:class} and \ref{thm:balmer}. For finite groups the recent paper \cite{CI} gives a different proof of the classification without a recourse to infinite dimensional modules whereas for finite groups schemes in general the theory of $\pi$-points  is so far the only tool available to prove Theorem~\ref{thm:class}.  

Another application of a good ``support variety theory"  for finite group schemes is a geometric criterion to determine the representation type.   Namely,  we have the following theorem:

\begin{theorem}[\cite{Far}, \cite{FaS}] Let $G$ be a finite group scheme. If $\dim \Spec \HHH^\bu(G,k) \geq 3$, then  
the representation theory of $G$ is wild.  
\end{theorem}

%%%%%%%%%%%%%%%%%%%%%%%%%%%%%%%%%%%%%%%%SECTION ON JORDAN TYPE%%%%%%%%%%%%%%%%%%%%%%%%%%%%%%%%%%%%%%

\section{Local Jordan type}
\label{sec:jordan}%non-maximal rank varieties, AR components, cohomological realization (gen support varieties paper)}
\begin{definition} Let $M$ be a finite dimensional $kG$-module, and  
let $\alpha_K: K[x]/x^p \to KG$ be a $\pi$-point. The {\it Jordan type} of $M$ at the $\pi$-point $\alpha_K$  
is the Jordan canonical form of $\alpha_K(x)$  considered as an operator on $M_K$. 
\end{definition}
Note that $\alpha_K(x)$ is a $p$-nilpotent operator, that is, $\alpha_K(x)^p =0$. Hence, the eigenvalues are all zero, and the  only Jordan blocks possible are the ones of size from $1$ to $p$. We use the exponential notation 
\[ \JType(\alpha_K, M) = [p]^{a_p}\ldots [1]^{a_1}\]
for the Jordan type where $a_i$ is the number of blocks of size $i$  in the Jordan form of $\alpha_K(x)$. 
An equivalent way of thinking about $\JType(\alpha_K, M)$ is that it is the isomorphism type of the $K[x]/x^p$-module $\alpha^*(M_K)$. 
We often refer to $\JType(\alpha_K, M)$ as a {\it local Jordan type} of $M$.  

We start with a well-known example of  what the local Jordan type of a module, considered at all $\pi$-points simultaneously, can determine.  The following theorem can be viewed as a generalization of the famous ``Dade's lemma" for elementary abelian $p$-groups \cite{D}.

\begin{theorem} 
Let $G$ be a finite group scheme, and $M$ be a finite dimensional $kG$-module.  
Then $M$ is projective if and only if $\JType(\alpha, M) = [p]^{\frac{\dim M}{p}}$ 
for all $\pi$-points $\alpha:k[x]/x^p \to kG$ defined over  $k$. 
\end{theorem} 

We refer the reader to D. Benson's article in the same proceedings for many examples and properties of local Jordan types of modules. We point out (as Benson does too) that in general the Jordan type is  NOT independent of a representative of an equivalence class of $\pi$-points. It is independent, though, if the  $\pi$-point is generic.   We say that $\alpha_K: K[x]/x^p \to KG$ is a {\it generic} $\pi$-point if the image of $\alpha_K$ under the map $\Psi$ of Theorem~\ref{thm:main}  is a generic point of $\Proj \HHH^\bu(G,k)$. The independence of $\JType(\alpha_K, M)$ of a representative of an equivalence class of $[\alpha_K]$ in this case is the most difficult part of the following theorem.

\begin{theorem}[\cite{FPS}] \label{thm:functor}
Let $G$ be a finite group scheme, and $\alpha_K: K[x]/x^p \to KG$ be a generic $\pi$-point of $G$. 
\begin{enumerate} \item  The functor 
\[\alpha^*_K: \stmod G \to \stmod K[x]/x^p \] defined by sending a $kG$-module $M$ to $\alpha^*_K(M)$ is an exact functor which commutes with tensor products.  
\item For two different choices of representatives $\alpha_K$, $\beta_L$ of the equivalence class $[\alpha_K]$, the corresponding functors are naturally isomorphic. 
\end{enumerate} \end{theorem}

%We now return to the world of finite dimensional modules. 
For a finite dimensional $kG$-module $M$ and a $\pi$-point $\alpha_K$ we denote by $\rk^j\{\alpha_K, M\}$ the rank of $\alpha_K(x^j)$ as an operator on $M_K$. 

\begin{definition}
\label{defn:rank} Let $M$ be a finite dimensional $kG$-module. We say that $M$ has {\it constant $j$-rank} if $\rk^j\{\alpha, M\}$   is independent of the choice of the $\pi$-point $\alpha$ defined over $k$.
\end{definition} 
Since $k$ is assumed to be algebraically closed, this definition is equivalent to requiring that $\rk^j\{\alpha_K, M\}$ is independent of $\alpha_K$ for any $\pi$-point $\alpha_K$ and any field extension $K/k$.  

Let $M$ be a $k[x]/x^p$-module, and let $\JType(x, M) = [p]^{a_p}\ldots[1]^{a_1}$   be the Jordan canonical form  of $x$ considered as an operator on $M$. The relations between the exponents $a_i$ and  the ranks of $x^j$ as operators on $M$ can be expressed explicitly with the following formulas: 
\begin{equation} 
\rk\{x^j, M\} = a_{j+1} + \ldots + a_p.
\end{equation} 
Therefore, we can give two equivalent definitions of modules of constant Jordan type: 
\begin{definition}
A finite dimensional $kG$-module $M$ is a module of {\it constant Jordan type} if it has constant $j$-rank for all $j$, $1 \leq j \leq p-1$.  Equivalently, $M$ has the same  Jordan type $\JType(\alpha, M)$  for any  $\pi$-point $\alpha$ defined over $k$.  In this case we refer to $\JType(\alpha, M)$ as Jordan type of $M$. 
\end{definition} 

Let $\alpha: k[x]/x^p \to kG$ be a $\pi$-point. Just as the Jordan type $\JType(\alpha, M)$, the rank of $\alpha(x)$ as an operator on $M$ is not well defined on an equivalence class of $\pi$-points in general.  But it is well defined when the rank is maximal. 

\begin{theorem}[\cite{FPS}]\label{thm:indep}  Let $M$ be a finite dimensional $kG$-module, and let $\alpha: k[x]/x^p \to kG$ be a $\pi$-point such that $\rk^j(\alpha, M)$ is maximal among $\rk^j(\beta,M)$ for all $\pi$-points $\beta: k[x]/x^p \to kG$. Then  for any $\pi$-point $\alpha^\prime$ equivalent to $\alpha$, $\rk^j(\alpha^\prime, M) = \rk^j(\alpha, M)$.
\end{theorem} 

\begin{remark}  A careful reader may notice a discrepancy between the definition of modules of constant Jordan type  given here and  in D. Benson's article.  Benson chooses generators of the augmentation ideal of $kE$ for an elementary abelian $p$-group $E$, and then considers Jordan type of  linear combinations of those generators.  Our definition in terms of $\pi$-points is ``generator-independent" and might appear to be a stronger condition.  Thanks to Theorem~\ref{thm:indep} these two approaches are equivalent. Indeed, if we choose a different set of generators of the augmentation ideal of $kE$, we do not change the equivalence class of a $\pi$-point as seen in Example~\ref{equiv:elem}. Hence, Theorem~\ref{thm:indep} implies that the property of constant rank, and, hence, constant Jordan type, does not depend upon the choice of generators.  
\end{remark}

Using local Jordan type, we can define new geometric invariants of modules. 
\begin{definition}
Let $G$ be a finite group scheme, and $M$ be a finite dimensional $kG$-module. The {\it non-maximal} $j$-rank  variety  of $M$ is the following subset of $\Pi(G)$:
\[\Gamma^j(G)_M = \{ [\alpha_K]\in \Pi(G) \; | \,  \rk^j(\alpha_K, M) \text{ is not maximal} \}. \]
The set  
\[\Gamma(G)_M = \bigcup\limits_{j=1}^{p-1} \Gamma^j(G)_M\]  
is called the {\it non-maximal rank variety of $M$}.
\end{definition} 

\begin{proposition}[\cite{FP10}]
\begin{enumerate} \item
The varieties $\Gamma^j(G)_M$ are closed proper subvarieties of $\Pi(G)$.
\item $\Gamma(G)_M^j = \emptyset$ if and only if $M$ is a module of constant $j$-rank. 
\item $\Gamma(G)_M = \emptyset$ if and only if $M$ is a module of constant Jordan type.  
\end{enumerate} 
\end{proposition}

We refer the reader to \cite{CFP} and \cite{FP10} for further properties and examples of non-maximal rank varieties.    
We mention here one of their obvious virtues: the non-maximal variety is always  properly contained in the ambient space $\Pi(G)$.   Hence, it carries non-trivial information for any module - in particular, for all those modules for which the support variety coincides with $\Pi(G)$. 
Before moving on to ``global invariants" in the next section we point out one important property that non-maximal rank varieties share with support varieties: they are invariants of the connected components of the  stable Auslander-Reiten quiver of $G$.

\begin{theorem}[\cite{FP10}] 
Let G be a finite group
scheme. Let $\Theta$ be a connected component of the stable Auslander-Reiten quiver of $kG$. Then for any two indecomposable modules $M$, $N$, belonging to $\Theta$ and any $j$, $1\leq j \leq p-1$, 
$\Gamma^j(G)_M = \Gamma^j(G)_N$.
\end{theorem} 

%%%%%%%%%%%%%%%%%%%%%%%%%%%%%%%%%%%%%%%%%%%SECTION ON $\THETA$%%%%%%%%%%%%%%%
\section{Global $p$-nilpotent operator  and vector bundles on $\bP(G)$}
\label{sec:theta}
%{sheaves and vector bundles  (get the statement right about, add the inverse), some examples from the construction paper.}
We now specialize to the case of an infinitesimal group scheme $G$ (with a side note that elementary abelian $p$-groups are included in consideration by virtue of Example~\ref{ex:elem}).  

Let  $G$ be an infinitesimal group scheme of height $r$.  Recall that the functor of ``one-parameter subgroups" of $G$, $V(G)$, is representable by the coordinate algebra $A=k[V(G)]$, so that we have a natural isomorphism (\ref{thm:repr})
\begin{equation} 
\label{eq:natur} 
\Hom_{\rm grp. sch}( \bG_{a(r),R}, G_R)  \simeq \Hom_{k-\rm alg}(A, R)
\end{equation}   
for any commutative $k$-algebra $R$.   Taking $R = A$,   we   get an isomorphism 
\begin{equation}\label{eq:univ} \Hom_{\rm grp. sch}( \bG_{a(r),A}, G_A)  \simeq \Hom_{k-\rm alg}(A, A).\end{equation} 
\begin{definition}
The {\it universal one-parameter subgroup} of $G$ is the one-parameter subgroup 
\[\cU: \bG_{a(r),A} \to G_A\]
which corresponds to the identity map $\id_{A}: A \to A$ via the isomorphism \eqref{eq:univ}.
\end{definition}
By naturality,  any one-parameter subgroup $\mu: \bG_{a(r),R} \to G_R$ can be obtained from $\cU: \bG_{a(r),A} \to G_A$ via base change $\mu=\cU \otimes_A R$ where $R$ is given the structure of an $A$-module via the map $f_\mu: A \to R$ corresponding to $\mu$  in \eqref{eq:natur}. 

The universal one-parameter subgroup $\cU$ induces an $A$-linear homomorphism of coordinate algebras: 
\[\cU^*: k[G] \otimes A \to k[\bG_{a(r)}] \otimes A\] 
Dualizing, we get an $A$-linear homomorphism 
\[\cU_*: k\bG_{a(r)} \otimes A\to kG \otimes A.\]
For the following definition we utilize   the map 
$$
\epsilon: k[x]/x^p \to k\bG_{a(r)}\simeq k[u_0, \ldots, u_{r-1}]/(u_0^p, \ldots, u_{r-1}^p)
$$ of Notation~\ref{not:epsilon}. 
\begin{definition}   
The {\it global $p$-nilpotent operator} 
$$
\Theta \in kG \otimes k[V(G)] = kG \otimes A
$$ is defined as the image of the generator $x$ under the composition 
\[\xymatrix{k[x]/x^p \ar[r]^-\epsilon & k\bG_{a(r)} \ar@{^(->}[r] & k\bG_{a(r)} \otimes A  \ar[r]^-{\cU_*} & kG \otimes A,}\]
that is, \\[2pt]
\centerline{\fbox{$\Theta = \cU_*(\epsilon(x)\otimes 1)$.}}  
\end{definition}   	

\begin{notation}
\label{not:theta} Let $\nu$ be a point of $V(G)$ with the residue field $k(\nu)$ (hence, $\nu$ is given by a map $f_\nu: A \to k(\nu)$). We denote by 
\[\theta_\nu = \Theta \otimes_A k(\nu)\] 
the specialization of $\Theta \in kG \otimes A$ at the  point $\nu \in V(G)$.
\end{notation}
Let $\mu_{k(\nu)}: \bG_{a(r), k(\nu)} \to G_{k(\nu)}$  be the one-parameter subgroup that corresponds to the point $\nu \in V(G)$, and let 
$\alpha_{k(\nu)} = 
 \mu_{k(\nu), *} \circ\epsilon: k(\nu)[x]/x^p \to k(\nu)G$  be the $\pi$-point defined by the one-parameter subgroup $\mu_{k(\nu)}$ as in \eqref{eq:pi}. 
Then  
\[\alpha_{k(\nu)}(x) = \Theta \otimes_{A} k(\nu) = \theta_\nu, \]
and for any finite dimensional $kG$-module $M$, we have 
\[ \JType(\alpha_{k(\nu)}, M) = \JType(\theta_\nu, M).\]

%the specialization of the universal $p$-nilpotent operator via the map $f_K: A\to K$.  

For a $kG$-module $M$, $\Theta$ determines  a $p$-nilpotent $A$-linear operator: 
\[\Theta_M: M \otimes A \to M \otimes A \]
given by the formula
\[ m \otimes f \mapsto \Theta \cdot (m \otimes f) \]
We give some explicit examples of $\Theta$.   

\begin{example}
\label{ex:oper_Lie}
Let  $\fg$  be a restricted Lie algebra, and $G$ be the corresponding infinitesimal group scheme of height 1.  
Then $kG  \simeq \fu(\fg)$, and the reduced scheme $V(G)_{\rm red}$ corresponds to the affine variety $\cN_p$. We have therefore a projection map $kG \otimes k[V(G)] \to kG \otimes k[V(G)_{\rm red}] \simeq kG \otimes k[\cN_p] \simeq \fu(\fg) \otimes k[\cN_p]$. 
Abusing notation slightly, we consider the global $p$-nilpotent element $\Theta$ as an element in $\fu(\fg) \otimes k[\cN_p]$, effectively  factoring out the nilpotents in $k[V(G)]$. 

Let $x_1, \ldots, x_n$  be a basis of $\fg$; and let $y_1, \ldots, y_n$  be the dual basis of $\fg^\#$.  The embedding 
$\cN_p \subset \fg$ induces a surjective map of algebras $S^*(g^\#)= k[y_1, \ldots, y_n] \twoheadrightarrow k[\cN_p]$. We denote by   
$\{Y_1, \ldots, Y_n\}$ the images of $\{y_1, \ldots, y_n\}$ in $k[\cN_p]$. With this notation, we have the following formula  
for the global $p$-nilpotent operator 
$$\Theta 
= x_1 \otimes Y_1 + \ldots + x_n \otimes  Y_n 
\in u(\fg) \otimes k[\cN_p]$$
Any nilpotent element $x = \lambda_1x_1+\ldots + \lambda_nx_n \in \cN_p$ is a specialization of $\Theta$ for some values $(\lambda_1, \ldots, \lambda_n)$ of $(Y_1, \ldots, Y_n)$. 

For a  restricted $\fg$-module $M$, the action of the operator $\Theta_M: M \otimes k[\cN_p] \to M \otimes k[\cN_p]$ is given by the following formula: 
\[\Theta_M(m \otimes f) = \sum\limits_{i=1}^n x_im \otimes Y_if. \]
\end{example}

\begin{example} 
\label{ex:oper_elem} Take $\fg = \fg_a^{\oplus r}$ so that $\fu(\fg) \simeq kE$ for $E$ an elementary abelian $p$-group of rank $r$.  Then $\cN_p \simeq \bA^r$ is the affine $r$-space, and $k[\bA^r] \simeq k[Y_1, \ldots, Y_r]$. The global operator $\Theta$ is given by the same formula as in Example~\ref{ex:oper_Lie}: 
\[\Theta = x_1 \otimes Y_1 + \ldots + x_r \otimes Y_r.\] 
Hence, in the case of an elementary abelian $p$-group our operator $\Theta$ is given by the same formula as the operator $\theta$ in Benson's article in this volume, even though for historical reasons slightly  different notation is used.    
\end{example}
For infinitesimal group schemes of higher height the formulas for $\Theta$ become much more complicated. We refer the reader to multiple examples worked out in \cite[\S 2]{FP12}.  
 
\begin{theorem}[1.11 \cite{SFB1}, 2.10, 2.11 \cite{FP12}] 
Let $G$ be an infinitesimal group scheme of height $r$. The coordinate algebra $k[V(G)]$ is a graded connected $k$-algebra generated by homogeneous elements in degrees $\{1, p, \ldots, p^{r-1}\}$.  The global $p$-nilpotent operator $\Theta \in kG \otimes k[V(G)]$ is homogeneous of degree $p^{r-1}$ where the grading on $kG \otimes k[V(G)]$ is induced by the grading on $k[V(G)]$ with $kG \otimes k$ being assigned degree $0$.   
\end{theorem}   

Since $k[V(G)]$ is graded, there is a corresponding projective scheme $\Proj V(G)$. Let 
\[\bP(G) = \Proj k[V(G)]_{\rm red}\]
be the associated reduced projective scheme, and let $\cO_{\bP(G)}$ be the structure sheaf.  
Since $k[V(G)]$ is generated in degrees dividing $p^{r-1}$, the Serre twist $\cO_{\bP(G)}(p^{r-1})$ 
is a locally free sheaf on $\bP(G)$ (see \cite[4.5]{FP12}).  The homogeneity  of $\Theta$ implies that for any $kG$-module $M$, $\Theta_M$ induces a sheaf homomorphism
\[\xymatrix@=7mm{\wt \Theta_M: M \otimes \cO_{\bP(G)} \ar[r] & M \otimes \cO_{\bP(G)}(p^{r-1}).}\] 
We define several sheaves on $\bP(G)$ associated to a $kG$-module $M$ via the operator $\wt \Theta_M$. 
Let 
\[\wt M = M \otimes \cO_{\bP(G)}\]
denote the trivial vector bundle defined by the module $M$ (equivalently, free $\cO_{\bP(G)}$-module of rank $\dim M$). 
We can iterate $\wt \Theta_M$ once we twist it appropriately.  To avoid cumbersome notation, we use $\xymatrix@=6mm{\wt\Theta^j_M: \wt M \ar[r] & \wt M(jp^{r-1})}$ to denote the composition 
\[\xymatrix@=8mm{\wt M \ar[r]^-{\wt \Theta_M} & \wt M(p^{r-1}) \ar[rr]^-{\wt \Theta_M(p^{r-1})}&&  \wt M(p^{r-1}) \ar[rr]^-{\wt \Theta_M(2p^{r-1})} &&\ldots\ldots 
\ar[rr]^-{\wt \Theta_M((j-1)p^{r-1})} 
&& \wt M(jp^{r-1})}.\]  
An easy way to follow our convention is to keep in mind that all the sheaves we define should be subs, quotients  or subquotients of $\wt M$ (no twists).
\begin{definition}   Let $j$ be an integer, $1 \leq j \leq p$.
\[\cKer \wt \Theta^j_M := \cKer \{\wt\Theta^j_M: \wt M \to \wt M(jp^{r-1})\}\] 
\[\cIm \wt \Theta^j_M := \cIm \{\wt\Theta^j_M: \wt M(-jp^{r-1}) \to \wt M\}\]
\[\cCoker \wt \Theta^j_M  := \cCoker \{\wt\Theta^j_M: \wt M(-jp^{r-1}) \to \wt M\}\]
\[\cF_j(M) := \frac{\cKer \wt \Theta_M \cap 	 \cIm \wt \Theta^{j-1}_M}{\cKer \wt \Theta_M \cap 	 \cIm \wt \Theta^j_M}.\] 
\end{definition} 

\begin{theorem}[5.1 \cite{FP12}]\label{thm:bundles} 
Let $G$ be a finite group scheme, and let $M$ be a finite dimensional module of constant $j$-rank for some $j$, $1 \leq j \leq p$.  Then $\cKer \wt \Theta^j_M$, $\cIm \wt \Theta^j_M$, $\cCoker \wt \Theta^j_M$ are algebraic vector bundles (equivalently, locally free coherent sheaves) on $\bP(G)$. 
\end{theorem} 
\begin{remark} It was mistakenly stated in \cite[4.13, 5.1]{FP12} that Theorem~\ref{thm:bundles} held ``if and only if". This mistake was pointed out to the author by J. Stark, see \cite{St2} for counterexamples. The following statement, though, is true.
\end{remark}   

\begin{theorem}   Let $G$ be a finite group scheme, and $M$ be a finite dimensional $kG$-module. The following are equivalent: 
\begin{enumerate}
\item $M$ is a module of constant Jordan type; 
\item For each $i$, $ 1 \leq i \leq p$, $\cF_i(M)$ is an algebraic vector bundle on $\bP(G)$. 
\end{enumerate}
\end{theorem}
\begin{proof} Even though  formulated for elementary abelian $p$-groups, the proof in {\cite[7.4.12]{Ben2}} goes through without change for any finite group scheme. See also \cite[3.9]{St2}.
\end{proof}
The sheaves $\cKer \wt \Theta^j_M$, $\cIm \wt \Theta^j_M$, $\cCoker \wt \Theta^j_M$, $\cF_i(M)$ provide global geometric invariants of representations of $G$ which carry more information than  support varieties or local Jordan type.  For a simple example, we show that the sheaves associated to representations distinguish between dual modules whereas local Jordan type is always the same for $M$ and $M^\#$. 
\begin{example}
 Let $G = \bG_{a(1)} \times \bG_{a(1)}$ so that $kG \simeq k[x,y]/(x^p, y^p)$.   
Let $W_n$  be  the $(2n+1)$--dimensional $kG$-module represented by the following diagram: 
$$
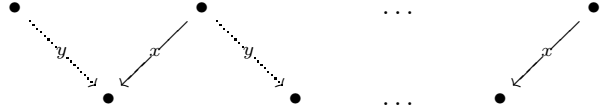
\begin{xy}*!C\xybox{%
\xymatrix{ {\bu} 
\ar@{.>}[dr]|y &&
{\bu} \ar[dl]|x \ar@{.>}[dr]|y
&&\dots &&{\bu} \ar[dl]|x\\
& \bu && \bu &\dots &  \bu &}}
\end{xy}
$$
In the diagram the dots represent the basis of $W_n$ as a $k$-vector space; the solid arrows give the action of $x$ and the dotted arrows give the action of $y$. The action of both generators on the bottom row is trivial.

Both $W_n$ and its dual $W_n^\#$ have constant  Jordan type $n[2] +  [1]$  
(see \cite[\S 2]{CFP}). We have $\bP(G) = \bP^1$, and the vector bundles 
$\cF_1$ are different for $W_n$ and  $W_n^\#$: 
\[ \cF_1(W_n) \simeq  \cO_{\bP^1}(-n) , \quad \cF_1(W_n^\#) \simeq \cO_{\bP^1}(n). \] 
\end{example}
There is  a general duality: 
\begin{proposition}[5.5 \cite{FP12}] Let $G$ be an infinitesimal group scheme, and
let $M$ be a finite dimensional $kG$-module of constant $j$-rank.  There is an isomorphism of $\cO_{\bP(G)}$-modules:
\[(\cKer \wt \Theta^j_M)^\vee \simeq \cCoker \wt \Theta^j_{M^\#}\]
where $(-)^\vee$ denotes the dual vector bundle. 
\end{proposition}

Even though finer invariants than support varieties and local Jordan types, the functors $M \mapsto \cKer\wt \Theta^j_M$ and others are by no means faithful.  In particular, they vanish on the interesting class of modules with constant kernel/image property studied, in particular, in \cite{CFS}. 

\begin{definition}{\cite[5.11]{FP12}} 
Let $M$ be a finite dimensional $kG$-module. We say that $M$ has {\it constant $j$-image property} if there exists a subspace $I(j) \subset M$ such that for any point $\nu \in V(G)$, $\nu \not = 0$, 
\[\Im\{\theta^j_\nu: M_{k(\nu)} \to  M_{k(\nu)}\} = I(j)_{k(\nu)}.\] 

Similarly, $M$ has {\it constant $j$-kernel property} if there exists a subspace $K(j) \subset M$ such that for any point $\nu \in V(G)$, $\nu \not = 0$, 
\[\Ker\{\theta^j_\nu: M_{k(\nu)} \to  M_{k(\nu)}\} = K(j)_{k(\nu)}.\]
\end{definition}

\begin{proposition}[5.12 \cite{FP12}]  
Let $M$ be a $kG$-module of constant $j$-rank.  
Then the algebraic vector bundle $\cIm \wt \Theta_M^j$ is trivial (i.e., a free coherent sheaf) 
 if and only if $M$ has constant $j$-image property. Similarly, $\ \cKer \wt \Theta_M^j$ is trivial if and only 
if $M$ has constant $j$-kernel property.
\end{proposition}

Let $\cH^{[1]}(M) = \frac{\cKer \wt \Theta^{p-1}_M}{\cIm \wt \Theta_M}$.  Recall that a $kG$-module $M$ is endotrivial if 
$\End_k(M) \simeq k$ in $\stmod G$.    It was shown in \cite[\S 5]{CFP} that  an endotrivial 
module is a  module of constant  Jordan type with possible types $[p]^{a_p}[1]$ and $[p]^{a_p}[p-1]$.  
One can also characterize endotrivial modules in terms of the sheaf $\cH^{[1]}(M)$. 

\begin{proposition}[5.17 \cite{FP12}]
\label{endo} Let $G$  be an infinitesimal group scheme, and assume that $G$ has a subgroup scheme isomorphic to $\mathbb G_{a(1)}\times\mathbb G_{a(1)}$ or $\mathbb G_{a(2)}$. Let $M$  be a module of constant Jordan type.
Then $\cH^{[1]}(M)$  is a   line bundle (i.e., an algebraic vector bundle of rank one) 
if and only  if $M$  is endotrivial.    \end{proposition}
In particular,  $\cH^{[1]}(-)$ induces a map from the group of endotrivial modules to the Picard group of $\bP(G)$. 
In \cite{Bal2}, Balmer constructs a map in the opposite direction (after inverting $p$) for $G$ a finite group. The  
connection between these two maps  is yet to be understood. 

\vspace{0.1in}
We give examples of some explicit calculations of bundles corresponding to $\fsl_2$-modules.  The nullcone $\cN(\fsl_2)$ is a quadric in $\bA^3$ defined by the equation $z^2+xy=0$. Hence, there is an isomorphism $i: \bP(\fsl_2) \simeq \bP^1$ (see, for example, \cite[5.8.1]{FP12}). In the calculations below we identify $\bP(\fsl_2)$ with $\bP^1$ via this isomorphism which allows us to use the standard notation for line bundles on $\bP^1$. 

Recall that  the representation theory of $\fu(\fsl_2)$ is tame.  
The complete list of indecomposable $\fsl_2$-modules consists of the following four families (see, for example, \cite{Pre2}), with some overlap: 
\begin{enumerate} 
\item  Weyl modules $V(\lambda)$;
\item  Induced modules $\HHH^0(\lambda) \simeq V(\lambda)^\#$;  
\item  Modules $\Phi_\xi(\lambda)$ for $\xi \in \bP^1_k$. These modules do not generally have rational $\SL_2$ structure and do not have constant Jordan type.  The projectivized support variety  of $\Phi_\xi(\lambda)$ contains one point: $\xi$.
\item Projective indecomposable modules $Q(\lambda)$ for $0 \leq \lambda \leq p-1$, with $Q(p-1) = V(p-1)$  being the irreducible Steinberg module. 
\end{enumerate}  It is shown in \cite{St} that for any finite dimensional restricted $\fsl_2$-module $M$, $\cKer \wt \Theta_M$ is a vector bundle, whether $M$ has constant rank or not.  
\begin{proposition}[\cite{St}] 
Let $\lambda = rp + a$, $ 0 \leq a \leq p-1$, $r \geq 0$    
\begin{enumerate}
\item $\cKer \wt \Theta_{V(\lambda)} \simeq \cO_{\bP^1}(-\lambda) \oplus \cO_{\bP^1}(a+2-p)^{\oplus r}$;
\item $\cKer \wt \Theta_{\HHH^0(\lambda)} \simeq \cO_{\bP^1}(-a)^{\oplus r+1}$;
\item $\cKer \wt \Theta_{\Phi_\xi(\lambda)} \simeq \cO_{\bP^1}(a+2-p)^{\oplus r}$;
\item $\cKer \wt \Theta_{Q(a)} \simeq \cO_{\bP^1}(-\lambda) \oplus \cO_{\bP^1}(a+2-2p)$.
\end{enumerate} 

\end{proposition} 

\begin{remark}\label{finite} 
The construction of the global operator $\wt \Theta$ uses properties of infinitesimal group schemes in an essential way. 
We can transfer the resulting constructions of sheaves and algebraic vector bundles to the variety $\Proj \HHH^\bu(G,k)$ via the isogeny 
$\Psi$  of Theorem~\ref{thm:inf} but the structure of the variety of  one-parameter subgroups $V(G)$ is crucial for the very definition of $\Theta$. It is an open question whether an analogous construction exists for finite groups (apart from elementary abelian which are really just restricted Lie algebras in disguise).
\end{remark}

%%%%%%%%%%%%%%%%%%%%%%%%%%%%%%%%%%%%%%SECTION ON ELEMENTARY SUBALGEBRAS%%%%%%%%%%%%%%%%%%%%%%%%%%%%%%%%%%%%%%%%%
\section{Elementary subalgebras of restricted Lie algebras}
\label{sec:elem}
In this section we consider only infinitesimal group schemes of height one or, equivalently, restricted Lie algebras. As usual, all considerations apply to an elementary abelian $p$-group $E$ of rank $r$ once we choose an isomorphism $kE \simeq \fu(\fg_a^{\oplus r})$. 
\begin{definition} A subalgebra $\epsilon$ of a restricted Lie algebra $\fg$ is called {\it elementary} if $\epsilon$ is an abelian Lie algebra with trivial $[p]^{\rm th}$ power.  Equivalently, $\epsilon \simeq \fg_a^{\oplus r}$ for some integer $r>0$. We define 
\[ \bE(r,\fg) = \{ \epsilon \subset \fg \, | \, \epsilon\text{ elementary subalgebra of dimension } r \}. \]  
\end{definition}  
\noindent
Recall that for $\epsilon\simeq\fg_a^{\oplus r}$, we have 
$$\fu(\epsilon) \simeq k[x_1, \ldots, x_r]/(x_1^p, \ldots, x_r^p).$$ 
\noindent
For a vector space $V$ of dimension $n \geq r$, we denote by $\Grass(r, V)$ the Grassmannian of $r$-dimensional linear subspaces of $V$. 
\begin{proposition} Let $\fg$ be a restricted Lie algebra.  For any $r \leq \dim \fg$, there exists a closed embedding $\bE(r, \fg) \hookrightarrow \Grass(r, \fg)$  which gives $\bE(r, \fg)$ a structure  of a projective algebraic variety. 
\end{proposition} 
Except for some special cases, the varieties $\bE(r, \fg)$ are quite mysterious. In the case $r=1$ they are familiar objects though: we get projectivizations of restricted nullcones. Hence, $\bE(r,\fg)$ can be viewed as a natural generalization of the rank variety of $\fg$.  Constructions of rank varieties and generalized rank varieties for modules as well as of the global $p$-nilpotent operator $\Theta$ and the associated sheaves all carry through for this new ambient variety $\bE(r,\fg)$.   Hence, we get a whole slew of new invariants of modular representations as well as new techniques for constructing algebraic vector bundles on interesting projective varieties.

\begin{example} Let $r=1$. Then $\bE(1, \fg) \simeq \Proj k[\cN_p]$. In particular, if $G$ is a reductive connected algebraic group and $\fg = \Lie G$, then $\bE(1, \fg)$ is irreducible (\cite{NPV}, \cite{UGA}).
\end{example} 

\begin{example}  Let $r=2$. Let $G$ be a connected reductive algebraic group, let $\fg = \Lie G$, and assume that $p$ is good for $G$. A result of  Premet (\cite{Pre2}) implies that $\bE(2, \fg)$ is equidimensional;  in the special case  $\fg = \gl_n$, 
$p \geq n$,  $\bE(2, \gl_n)$ is irreducible of dimension $n^2 -5$. 
\end{example}
Things get murky once $r \geq 3$. On the other hand, there is quite a bit we can say if we consider maximal elementary subalgebras of a given Lie algebra $\fg$. 
\begin{notation} Let $\fg$ be a restricted Lie algebra. Then 
\[\rk_{el}(\fg) = \max \{r \, | \, \text{there exists an elementary subalgebra } \epsilon \subset \fg \, : \, 
\dim \epsilon =r \}. \]
\end{notation} 
\noindent
For complex simple Lie algebras the dimension of a maximal 
abelian subalgebra was determined by Malcev in 1945 \cite{Mal51}; 
the general linear case was first considered by 
Schur at the turn of the last century \cite{Sch05}. We give some examples of calculations 
of $\rk_{el}(\fg)$ and the corresponding varieties $\bE(r, \fg)$ for restricted Lie algebras; more can be found  in \cite{CFP3}. 

We denote by $\Grass(r,n)$ the {\it Grassmannian} of $r$-planes in $n$-space, and by $\LG(n, n)$ the {\it Lagrangian Grassmannian} of isotropic subspaces of maximal dimension in a $2n$-dimensional symplectic space.  

\begin{theorem}\label{thm:sl}
 Let $\fg = \gl_n$ (or $\fsl_n$).   
\begin{enumerate} 
\item If $ n = 2m$, then $\rk_{el}(\fg) = m^2$, and
$$\bE(m^2, \fg) \simeq \Grass(m, 2m).$$
\item  If $n = 2m+1$, $m \geq 2$, then $\rk_{el}(\fg) = m(m+1)$, and
$$\bE(m(m+1), \fg) \simeq \Grass(m, 2m+1) \sqcup \Grass(m, 2m+1).$$	
\end{enumerate}  
\end{theorem} 

\begin{theorem}\label{thm:sp}
For $\fg = \fsp_{2n}$, $\rk_{el}(\fsp_{2n}) = \frac{n(n+1)}{2}$, and
\[\bE\left(\frac{n(n+1)}{2}, \fsp_{2n}\right) \simeq \LG(n, n). \]
\end{theorem} 

\subsection*{Modules of constant $(r,j)$ rank and vector bundles on $\bE(r, \fg)$}
%The notion of a module of constant $j$-rank allows for a generalization for higher dimensional elementary subalgebras. 
For an elementary subalgebra $\epsilon \subset \fg$, and a $\fg$-module $M$, we denote by $M{\downarrow_{\epsilon}}$ the restriction of $M$ to $\epsilon$. 
Then $\Rad^j(M{\downarrow_{\epsilon}})$  is the $j^{\rm th}$ radical of $M$ as an $\fu(\epsilon) \simeq k[x_1, \ldots, x_r]/(x_1^p, \ldots, x_r^p)$-module, and  
$\Soc^j(M{\downarrow_{\epsilon}})$ is the $j^{\rm th}$ socle. 
%The following definition generalizes Definition~\ref{defn:rank}. 

\begin{definition} A restricted $\fg$-module $M$ is a module of {\it constant $(r,j)$-radical rank} if the dimension of $\Rad^j(M{\downarrow_{\epsilon}})$ is independent of $\epsilon \in \bE(r, \fg)$. 

We say that $M$ is a module of {\it constant $(r,j)$-socle rank} if the dimension of $\Soc^j(M{\downarrow_{\epsilon}})$ is independent of $\epsilon \in \bE(r, \fg)$.
\end{definition}

For $r=1$ the notions of constant $j$-radical and $j$-socle rank are equivalent. Moreover, they are both equivalent to the notion of constant $j$-rank. 
For $r>1$, the properties of constant $(r,j)$-radical and socle rank are independent of each other; one can find examples of this phenomenon in \cite{CFP2}.  
An analogue of modules of constant Jordan type for $r>1$ is given in the following definition: 
\begin{definition} 
A restricted $\fg$-module $M$ is a module of {\it constant $r$-radical type} if it is a module of constant $(r,j)$-radical rank for any integer $j > 0$. 

Similarly, 
$M$ is a module of {\it constant $r$-socle type} if it is a module of constant $(r,j)$-socle rank for any integer $j > 0$.
\end{definition}
Once again, for $r=1$ both notions are equivalent to constant Jordan type.  For $r>1$, properties of constant radical type and constant socle type  are independent. There are several constructions of modules of constant $(r,j)$-radical or socle rank for various combinations of $r$ and $j$ in \cite{CFP2}. Here, we give one family of examples which involves Carlson modules. 

For a positive degree cohomology class $\zeta \in \HHH^{m}(G,k)$,  the {\it Carlson module} $L_\zeta$ is the kernel of the map $\hat \zeta : \Omega^{m} k \to k$ that corresponds to $\zeta$ under the isomorphism $\HHH^m(G,k) \simeq \Hom_{\stmod G}(\Omega^m k ,k)$.    

\begin{theorem}[5.5 \cite{CFP2}] Let $E$ be an elementary abelian $p$-group of rank $n$, and let $\zeta \in \HHH^m(E,k)$ be a non-nilpotent positive dimensional cohomology class. If the hypersurface $Z(\zeta) \subset \Proj \HHH^\bu(E,k) \simeq \bP^{n-1}$ defined by the equation $\zeta = 0$  does not contain  a linear subspace of dimension $r-1$ then $L_\zeta$ has  constant $r$-radical  type. 
\end{theorem} 

Whereas modules $L_\zeta$ are never of constant Jordan type, for $r>1$ there are plenty of examples when the condition of the theorem is satisfied (see \cite[\S 3]{CFP2}).

\vspace{0.1in}
To construct algebraic vector bundles on $\bE(r,\fg)$ corresponding to representations of constant $(r,j)$-radical or socle rank, we need 
to replace the operator $\Theta$ for $r=1$ with a vector operator $(\Theta_1, \ldots, \Theta_r)$. The operators $\Theta_i$ are not defined globally  on $\bE(r,\fg)$, and there are two equivalent approaches to their construction that lead to associated kernel, image, and cokernel sheaves which {\it are} defined globally on $\bE(r, \fg)$: 
%There are two (equivalent) constructions for these operators, both are rather technically involved: \\[1pt]
 
\vspace{0.05in} 
(1) {\it Local construction}: this approach involves patching together images (respectively, kernels or cokernels)   of explicit 
linear maps on an affine covering of $\bE(r, \fg).$

\vspace{0.05in}
(2) {\it Equivariant descent}: this construction employs homogeneous linear operators  
$$\Theta_i: M \otimes k[\cN_p^r(\fg)] \to M \otimes k[\cN_p^r(\fg)][1]$$ 
for $i=1, \ldots, r$ 
where $\cN_p^r(\fg)$ is the variety of $p$-nilpotent commuting elements of $\fg$. The kernels and images of these operators are $\GL_r$-equivariant which allows to pass to sheaves on $\bE(r, \fg)$.\\[0.5pt]

Details can be found in \cite{CFP2}, \cite{CFP3}.  The outcome  of these somewhat technically involved constructions is the following theorem.
For a quasi-projective variety $X$ we denote by $\Coh(X)$ the category of coherent $\cO_X$-modules. 

\begin{theorem}[\S 5 \cite{CFP3}]\label{thm:functors} 
 Let $\fg$ be a restricted Lie algebra. There exist functors
$$\cIm^j, \cKer^j: \fu(\fg)-{\rm mod} \to \Coh(\bE(r, \fg))$$
such that the fiber of $\cIm^j(M)$ (respectively. $\cKer^j(M)$) for a  restricted $\fg$-module $M$ at a generic point 
$\epsilon \in \bE(r, \fg)$ is naturally identified with $\Rad^j({M\downarrow_{\epsilon}})$ (respectively, $\Soc^j(M{\downarrow_{\epsilon}})$). 
\sloppy{

}

More generally, for any locally closed subset $X \in \bE(r, \fg)$,  there exist functors 
$$\cIm^{j, X}, \cKer^{j, X}: \fu(\fg)-{\rm mod} \to \Coh(X)$$
such that the fiber of $\cIm^{j, X}(M)$ (respectively, $\cKer^{j, X}(M)$) at a generic point 
$\epsilon \in X$ is naturally identified with $\Rad^j(M{\downarrow_{\epsilon}})$ (respectively, $\Soc^j(M{\downarrow_{\epsilon}})$). 
\sloppy{

}
\end{theorem}  
 
For an algebraic group $G$, the variety $\bE(r, \fg)$ for $\fg = \Lie G$ comes equipped with an action of $G$. The second part of Theorem~\ref{thm:functors} allows one to construct sheaves on $G$-orbits of $\bE(r, \fg)$. 
Under this construction, modules of constant $(r,j)$ rank lead to algebraic vector bundles, in analogy with modules of constant $j$-rank.

\begin{proposition}[5.20, 5.21 \cite{CFP3}]
Let $M$ be a restricted $\fg$-module.  
If $M$ is a module of constant $(r,j)$-{\it radical} rank (respectively,  constant $(r,j)$-{\it socle} rank) then $\cIm^j(M)$  (respectively, $\cKer^j(M)$) is an {\it algebraic vector bundle} on $\bE(r,\fg)$ 
with fiber at $\epsilon \in \bE(r,\fg)$ naturally isomorphic to $\Rad^j(M\downarrow_\epsilon)$ (respectively, $\Soc^j(M\downarrow_\epsilon)$).
%[0.5pt] 

\vspace{0.05in}
More generally, let $X \subset \bE(r, \fg)$ be  a locally closed subset. If $\dim  \Rad^j(M\downarrow_\epsilon)$  (respectively, $\dim \Soc^j(M\downarrow_\epsilon)$) is constant for all $\epsilon \in X$, then $\cIm^{j,X}(M)$  (respectively, $\cKer^{j, X}(M)$) is an {\it algebraic vector bundle} on $X$  with fiber at $\epsilon \in X$ naturally isomorphic to $\Rad^j(M\downarrow_\epsilon)$ (respectively, $\Soc^j(M\downarrow_\epsilon)$). 
\end{proposition}
We finish with some computational examples of the kernel, image, and cokernel sheaves on $\bE(r, \fg)$.  
\begin{example} 
\label{ex:sl}  Let $\fg = \fsl_{2n}$, and let $V$ be the defining representation of $\fsl_{2n}$. Assume $p>2n-2$. We have $\bE(n^2, \fsl_{2n}) \simeq \Grass(n, 2n)$ by Theorem~\ref{thm:sl}. Let $\gamma_n$ be the canonical ($\sim$ tautological) rank $n$ vector subbundle on $\Grass(n, 2n)$.  Then
\begin{enumerate}
\item $\cIm^1(V) = \gamma_n$;
\end{enumerate}
For any $j \leq m$,
\begin{itemize}\item[(2)] 
$\cIm^{j}(V^{\otimes j})  \simeq \gamma_n^{\otimes j}$,
\item[(3)]
$\cIm^{j}(S^j(V)) \simeq S^j(\gamma_n),$
\item[(4)] 
$\cIm^{n}(\Lambda^n(V))  \simeq \Lambda^n(\gamma_n).$
\end{itemize}  
\end{example}
\noindent Similar calculations can be done for other Grassmannians $\Grass(r, n)$ by realizing them as $\GL_n$-orbits in $\bE(r(n-r), \fsl_n)$ (\cite[\S 6]{CFP3}).  

\begin{example}{\cite[6.9]{CFP3}} Let $\fg = \fsp_{2n}$, let $V$ be the defining representation for $\fsp_{2n}$, and assume that $p>3$. 
By theorem~\ref{thm:sp}, $\bE\left({{n+1} \choose 2}, \fsp_{2n}\right) = \LG(n,n)$, the Lagrangian Grassmannian.
We denote by $\gamma_n^{sp}$ the canonical ($\sim$ tautological) vector subbundle of rank $n$ on $\LG(n,n)$.  
%Similarly to the special linear case, this is the bundle (of rank m) with the fiber over the isotropic plane $W \subset V$ identifies with $W$.  
With this notation we have exactly the same results as  in  Example~\ref{ex:sl}:
\begin{enumerate}
\item $\cIm^1(V) = \gamma_n^{sp}$; 
\end{enumerate}
For any $j \leq n$,
\begin{itemize}\item[(2)] 
$\cIm^{j}(V^{\otimes j})  \simeq (\gamma_n^{sp})^{\otimes j}$,
\item[(3)]
$\cIm^{j}(S^j(V)) \simeq S^j(\gamma_n^{sp}),$
\item[(4)] 
$\cIm^{n}(\Lambda^n(V))  \simeq \Lambda^n(\gamma_n^{sp}).$
\end{itemize}  
\end{example}

\noindent In our last example we demonstrate that using representations of Lie algebras one can easily realize tangent and cotangent bundles on Grassmannians.  As with the examples above, it works very  similarly for the special linear and the symplectic case. 
\begin{example} 
Let $\fg = \fsl_{2n}$ (respectively, $\fg = \fsp_{2n}$). Consider $\fg$ acting on itself via the adjoint representation.   Let
$X = \bE({n^2},\fsl_{2n}) \simeq \Grass(n, 2n)$ (respectively, $X = \bE\left({{n+1} \choose 2},\fsp_{2n}\right) \simeq \LG(n, n)$). 
\begin{enumerate}
\item $\cCoker(\fg) \simeq T_X$, the tangent bundle on $X$, 
\item $\cIm^2(\fg) \simeq \Omega_X$, the cotangent bundle on $X$. 
\end{enumerate}
\end{example}
For other simple algebraic groups $G$ one can make analogous calculations for bundles on homogeneous spaces associated with cominuscule parabolics by considering  $G$-orbits   of $\bE(r, \Lie G)$ (see \cite{CFP3}).

%\section{References}

%\renewcommand{\refname}{}    %%%% for this example
%\vspace*{-36pt}              %%%% file only!


\frenchspacing
\begin{thebibliography}{7}

\bibitem{AE} J.L. Alperin, L. Evens, 
Representations, resolutions, and Quillen's dimension theorem, 
\textit{J. Pure $\&$ Applied Algebra} \textbf{22} (1981) 1--9. 

\bibitem{AJ} H. H. Andersen, J. C. Jantzen,
Cohomology of induced representations for algebraic groups,
\textit{Math. Ann.} \textbf{269} (1984), 487--525.




%\bibitem[\sf AB00]{AB00} L. Avramov, R. Buchweitz,
%Support varieties and cohomology over complete intersections,
%{\em Invent. Math.} 140 (2000), no. 1, 143--170.

\bibitem{AS} G. Avrunin, L. Scott, 
Quillen stratification for modules,
\textit{Invent. Math.} \textbf{66} (1982), 277--286.
%

\bibitem{Ba05}  P. Balmer, The spectrum of prime ideals in tensor triangulated categories,
\textit{J. f\"ur die Reine und Ang. Math. (Crelle)}, \textbf{588}, (2005), 149--168.

\bibitem{Bal2} ---------, Picard groups in triangular geometry and applications to modular representation theory, 
\textit{Trans. of the  AMS}, \textbf{362} (2010), no. 7, 1521--1563.


\bibitem{Bar} M. Barry, {Large Abelian Subgroups of Chevalley Groups}, \textit{J. Austral. Math. Soc.} (Series A) 
\textbf{27} (1979), 59--87.


\bibitem{Bendel} C. Bendel, Cohomology and projectivity of modules for finite group schemes, \textit{Math. Proc. Cambridge Philos. Soc.} \textbf{131} (2001), no. 3, 405--425.

\bibitem{BNPP} C.\ Bendel, D.\ Nakano, B.\ Parshall, and C.\ Pillen,
Cohomology for quantum groups via the geometry of the nullcone, to appear in \textit{Memoirs of the AMS}.



\bibitem{Ben} D. J. Benson, Representations and Cohomology, Vol I, II, Cambridge, (1998) 

\bibitem{Ben2} ---------, Representations of elementary abelian $p$-groups and vector bundles,
available at
homepages.abdn.ac.uk/mth192/pages/papers/b/benson/vecbook.pdf.

\bibitem{BCR}
D. J. Benson, J. F. Carlson, and J. Rickard, 
Complexity and varieties for infinitely generated modules, II, 
\textit{Math. Proc. Camb. Phil. Soc.} \textbf{120} (1996), 597--615.

\bibitem{Car} J. Carlson, The varieties and the cohomology ring of a
module, \textit{J. Algebra} \textbf{85} (1983) 104-143.


\bibitem{Car08} J. F. Carlson,  Rank varieties, \textit{Trends in representation theory of algebras and related topics}, {EMS Ser. Congr. Rep.}, Eur. Math. Soc., Zürich, (2008), 167–-200

\bibitem{CFP}  J. F. Carlson, E. M. Friedlander, J. Pevtsova,  {Modules of 
Constant  Jordan type}, \textit{Journal f\"ur die Reine und Angewandte 
Mathematik} \textbf{614} (2008), 191--234.

\bibitem{CFP2}  ---------, Representations of elementary 
abelian $p$-groups and bundles on Grassmannians, \textit{Advances in Math.} \textbf{229} (2012) 2985--3051.

\bibitem{CFP3}  ---------, Elementary subalgebras for modular Lie algebras, preprint (2012).

\bibitem{CFS} J. F. Carlson, E. M. Friedlander, A. Suslin, Modules for $\Z/p \times \Z/p$, \textit{Comment. Math. Helvetici} \textbf{86} (2011), 609--657.

\bibitem{CI} J. F. Carlson, S. Iyengar,
Thick subcategories of the bounded derived category of a finite group, preprint (2012)

\bibitem{D} E. C. Dade, 
Endo-permutation modules over $p$-groups, II,
\textit{Ann. Math.} \textbf{108} (1978), 317--346.


\bibitem{DG} M. Demazure, P. Gabriel, Groups Alg\'ebriques, tome I, Paris / Amsterdam (1970) \textit{Masson / North-Holland)}
\bibitem{EH} D. Eisenbud, J. Harris, The geometry of schemes. \textit{Graduate Texts in Mathematics}, \textbf{197}, Springer-Verlag, New York, 2000.

\bibitem{E} L. Evens, The cohomology ring of a finite group, 
\textit{Trans. A.M.S.} \textbf{101} (1961) 224--239.

\bibitem{EO} P. Etingof and V. Ostrik, Finite tensor categories,  \textit{Mosc. Math. J.} {\textbf 4} (2004), no. 3, 627--654, 782--783.


\bibitem{FW} C.G. Faith, E.A. Walker, \textit{Direct sum representations
of injective modules}, J. of Algebra  \textbf{5} (1967) 203-221.  

\bibitem{Far} R. Farnsteiner, Tameness and complexity of finite group schemes. 
 {\textit{Bull. Lond. Math. Soc.}} \textbf{39} (2007), no. 1, 63–70.
  
\bibitem{FaS} R. Farnsteiner,  A. Skowro\'nski, Classication of restricted Lie algebras with tame
principal blocks, \textit{J. Reine Angew. Math.} \textbf{546} (2002), 1--45.

\bibitem{FPar2} E. M. Friedlander, B. J. Parshall, 
Support varieties for restricted Lie algebras, 
\textit{Invent. Math.} \textbf{86} (1986), 553--562. 


 
\bibitem{FPar} E. M. Friedlander and B. J. Parshall,
Cohomology of Lie algebras and algebraic groups,
\textit{Amer. J. Math.} \textbf{108} (1986), 235-253.
%


\bibitem{FP05} E. M. Friedlander, J. Pevtsova,
Representation-theoretic spaces for finite group schemes, \textit{Amer. J. Math.} \textbf{127} (2005), 379--420.
%**

\bibitem{FP07} ---------,
$\Pi$-supports for modules for finite group schemes over a field,
\textit{Duke Math. J.} \textbf{139} (2007), no. 2, 317--368.
%**

\bibitem{FP10}  ---------, Generalized support varieties for finite group schemes, \textit{Doc. Math.}
(2010), Extra volume: Andrei A. Suslin sixtieth birthday, 197--222.


\bibitem{FP12} ---------, {Constructions for infinitesimal group schemes}, 
\textit{Trans. of the  AMS}, \textbf{363} (2011), no. 11, 6007-6061.




\bibitem{FPS}  E. M. Friedlander, J. Pevtsova, A. Suslin,  Generic and 
Maximal Jordan types, \textit{Invent. Math.} \textbf {168} (2007), 485--522. 


\bibitem{FS} E. M. Friedlander, A. Suslin, 
Cohomology of finite group schemes over a field,
\textit{Invent. Math.} \textbf{127} (1997), 209--270.  


\bibitem{GK} V.\ Ginzburg and S.\ Kumar, {Cohomology
of quantum groups at roots of unity}, \textit{Duke Math.\ J.} \textbf{69} (1993),
179--198.

\bibitem{Go} E. Golod, 
The cohomology ring of a finite $p$-group, (Russian) 
\textit{Dokl. Akad. Nauk SSSR} \textbf{125} (1959), 703--706. 

\bibitem{HP} M. Hovey, J.H. Palmieri, Stably thick subcategories of modules over Hopf algebras, \textit{Math.
Proc. Cambridge Philos. Soc.} \textbf{3} (2001), 441--474.



\bibitem{Jan}  J. C. Jantzen, {Representations of algebraic groups, $2^{nd}$ edition.}
\textit{Math Surveys and Monographs} {\bf 107}, AMS 2003.

\bibitem{Kel} B. Keller, Chain Complexes and Stable Categories, \textit{Manus. Math} \textbf{67} (1990) 379--417.

\bibitem{LS} R.G. Larson, M.E. Sweedler {An associative orthogonal bilinear form for Hopf algebras}, \textit{Amer. J. Math.} \textbf{91} (1967) 75--94.


\bibitem{Mal51} A. Malcev, {Commutative  subalgebras of semi-simple Lie algebras}, \textit{Bull. Acad. Sci. URSS} 
Ser. Math {\textbf 9}, (1945), 291-300 [Izvestia Akad. Nauk SSSR].  A.I. Malcev, {Commutative  subalgebras of semi-smple 
Lie algebras}, \textit{Amer. Math. Soc. Translation} \textbf{40} (1951).



\bibitem{Mc} G. McNinch,  Abelian unipotent subgroups of reductive groups \textit{J. Pure Appl. Algebra} \textbf{167} (2002), no. 2-3, 269–300. 

\bibitem{NPV} D. Nakano, B. Parshall, D. Vella, {Support varieties for algebraic groups}, 
\textit{J. Reine Angew. Math.} \textbf{547} (2002), 15--49.


\bibitem{MPSW} 
M. Mastnak, J. Pevtsova, P. Schauenburg,  S. Witherspoon, 
 Cohomology of finite-dimensional pointed Hopf algebras.
 \textit{Proc. Lond. Math. Soc.} (3) \textbf{100} (2010), no. 2, 377–404. 

\bibitem{Pe1}	J. Pevtsova,
 Infinite Dimensional Modules for Frobenius Kernels, 
 \textit{J. of Pure $\&$ Applied Algebra} \textbf{173} (2002), no 1, 59--83. 
%**



\bibitem{Pe2} J. Pevtsova, 
Support cones for infinitesimal group schemes,
\textit{Hopf Algebras}, Lect. Notes in Pure $\&$ Appl. Math., \textbf{237}, Dekker, New York, (2004), 203--213.

\bibitem{Pre2} A. Premet, {Nilpotent commuting varieties of reductive Lie algebras},
\textit{Invent. Math.}\textbf{154} (2003), 653-683.


\bibitem{Q} D. G. Quillen, 
The spectrum of an equivariant cohomology ring, I, II, 
\textit{Ann. Math.} \textbf{94} (1971), 549-572, 573-602.


\bibitem{Sch05} I. Schur, { Zur Theorie der vertauschbaren Matrizen,} \textit{J. Reine Angew. Math.} \textbf{130} (1905),  66--76.



\bibitem{Sob} P. Sobaje, On exponention and infinitsimal  one-parameter subgroups of reductive groups, , J. Algebra 385 (2013), pp. 14-26. 

\bibitem{St} J. Stark, Computations of sheaves associated to the representation theory of $\fsl_2$, preprint (2013), arXiv 1309:1505 .

\bibitem{St2} J. Stark, Detecting projectivity in sheaves associated to representations of infinitesimal groups, preprint (2013), arXiv 1309:1504 


\bibitem{SFB1} A. Suslin, E. M. Friedlander,  C. Bendel, 
Infinitesimal 1-parameter subgroups and cohomology,
\textit{J. Amer. Math. Soc.} \textbf{10} (1997), 693--728.
%

 
\bibitem{SFB2} A. Suslin, E. M. Friedlander, C. Bendel, 
Support varieties for infinitesimal group schemes,
\textit{J. Amer. Math. Soc.} \textbf{10} (1997), 729--759.

\bibitem{UGA} UGA VIGRE Algebra group, Varieties of nilpotent elements for simple Lie algebras I; Good Primes, 
\textit{J. Algebra}, {\textbf 280}, (2004), 719--737. Varieties of nilpotent elements for simple Lie algebras II; Bad Primes, 
\textit{Algebra(computational algebra section)},  \textbf{292}, (2005),  65--99.



\bibitem{V} B. B. Venkov, Cohomology algebras for some classifying spaces, 
\textit{(Russian) Dokl. Akad. Nauk SSSR} \textbf{127} (1959), 943--944. 




\bibitem{W} W. Waterhouse, {Introduction to affine group schemes}, 
\textit{Graduate Texts in Mathematics}, \textbf{66}, Springer-Verlag, New York-Berlin, 1979.


\end{thebibliography}
\end{document}